\documentclass[a4paper,leqno]{amsart}
\date{September 13, 2021}
\usepackage[active]{srcltx}
\usepackage{amsmath, amssymb, setspace}
\usepackage[dvipdfmx]{graphicx}
\usepackage{amssymb}
\title[ %
]{A generalization of Zakalyukin's lemma, and 
symmetries of surface singularities}
\author{A.~Honda}
\address[Atsufumi Honda]{
Department of Applied Mathematics, 
Faculty of Engineering, Yokohama National University,
79-5 Tokiwadai, Hodogaya, Yokohama 240-8501, Japan
}
\email{honda-atsufumi-kp@ynu.ac.jp}

\author{K.~Naokawa}
\address[Kosuke Naokawa]{%
Department of Computer Science, 
Faculty of Applied Information Science,
Hiroshima Institute of Technology,  
2-1-1 Miyake, Saeki, Hiroshima, 731-5193, Japan
}
\email{k.naokawa.ec@cc.it-hiroshima.ac.jp}

\author{K. Saji}
\address[Kentaro Saji]{%
  Department of Mathematics,
  Faculty of Science,
  Kobe University,
  Rokko, Kobe 657-8501}
\email{saji@math.kobe-u.ac.jp}
\author{M. Umehara}
\address[Masaaki Umehara]{%
  Department of Mathematical and Computing Sciences,
  Tokyo Institute of Technology,
  Tokyo 152-8552, Japan}
\email{umehara@is.titech.ac.jp}
\author{K. Yamada}
\address[Kotaro Yamada]{%
  Department of Mathematics,
  Tokyo Institute of Technology,
  Tokyo 152-8551, Japan}
\email{kotaro@math.titech.ac.jp}


\newcommand{\op}[1]{{\operatorname{#1}}}
\newcommand{\vect}[1]{{\boldsymbol{#1}}}

\newcommand{\R}{\boldsymbol{R}}

\newcommand{\mc}[1]{{\mathcal #1}}
\newcommand{\mb}[1]{{\mathbf #1}}

\newcommand{\pmt}[1]{{\begin{pmatrix} #1  \end{pmatrix}}}

\renewcommand{\phi}{\varphi}
\renewcommand{\epsilon}{\varepsilon}

\renewcommand{\det}{\op{det}}

\numberwithin{equation}{section}
\newtheorem{Theorem}{Theorem}[section]
\newtheorem{Proposition}[Theorem]{Proposition}
\newtheorem{Corollary}[Theorem]{Corollary}
\newtheorem{Lemma}[Theorem]{Lemma}
\newtheorem{Fact}[Theorem]{Fact}
\theoremstyle{definition}
\newtheorem{Def}[Theorem]{Definition}
\newtheorem{Remark}[Theorem]{Remark}
\newtheorem{Example}[Theorem]{Example}
\newtheorem*{acknowledgments}{Acknowledgments}

 \makeatletter
       \def\@makefnmark{%
               \leavevmode
               \raise.9ex\hbox{\check@mathfonts
                       \fontsize\sf@size\z@\normalfont%
                               \@thefnmark}%
       }
       \makeatother
\thanks{%
The first author was partially supported by 
Grant-in-Aid for Early-Career Scientists
 No.~19K14526 and No. 20H01801. 
The second author was partially supported by 
Grant-in-Aid for Young Scientists (B) No.~17K14197,
and the third author
was 
partially supported by  Grant-in-Aid for 
Scientific Research (C) No.\ 18K03301.
The forth author 
was partially 
supported by Grant-in-Aid for 
Scientific Research (B) No.\ 21H00981.
The fifth author 
was partially 
supported by Grant-in-Aid for 
Scientific Research (B) No.\ 17H02839.
}%

\begin{document}
\maketitle

\begin{abstract}
Zakalyukin's lemma asserts that the coincidence
of the images of two wave front germs
implies the right equivalence of corresponding map
germs under a certain genericity assumption.
The purpose of this paper is to give an improvement of this
lemma for frontals.
Moreover, we give several applications for
singularities on surfaces.
\end{abstract}

\section*{Introduction}
Let $p$ be a fixed point on the $n$-dimensional Euclidean
space $\R^n$ ($n\ge 1$) and $U$ a connected
neighborhood of $p$ in $\R^{n}$. 
In this paper, we set $r=\infty$ or $r=\omega$
and \lq\lq$C^r$'' means smoothness if $r=\infty$ and
real analyticity if $r=\omega$.

A $C^r$-map $f:U\to \R^{n+1}$ is called a {\it frontal} or
a {\it frontal map}  if $f$ admits a unit normal 
$C^r$-vector field $\nu$ defined on $U$.
By parallel transport in $\R^{n+1}$, the vector field 
$\nu$ can be identified
with its induced Gauss map $\nu:U\to S^n$, where $S^n$ is the
unit sphere centered at the origin of $\R^{n+1}$.
In this setting, the pair of $f$ and $\nu$ induces a
$C^r$-map (called the {\it Legendrian lift} of $f$)
$$
L_f:=(f,\nu):U\to \R^{n+1}\times S^n.
$$
If $L_f$ is an immersion, then $f$ is called a {\it wave front}.
In the case of $n=2$, cuspidal edges and swallowtails are
singular points appearing on wave fronts.
Germs of cuspidal cross caps are not wave fronts, 
but are  frontals. On the other hand, germs of cross caps are not frontals.

\begin{figure}[htb]
\begin{center}
  \includegraphics[height=2.0cm]{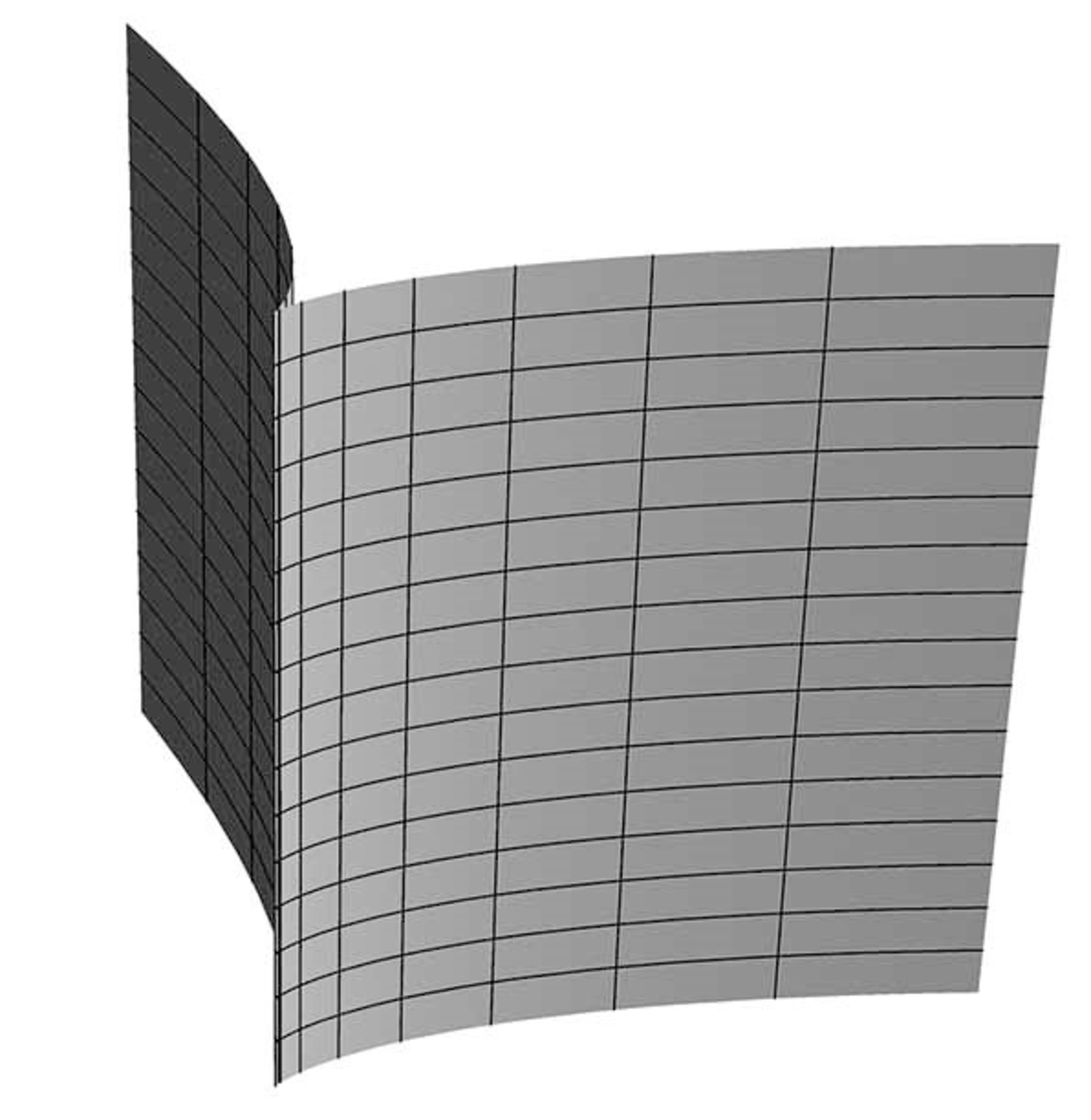}  
\,\,\, \includegraphics[height=1.5cm]{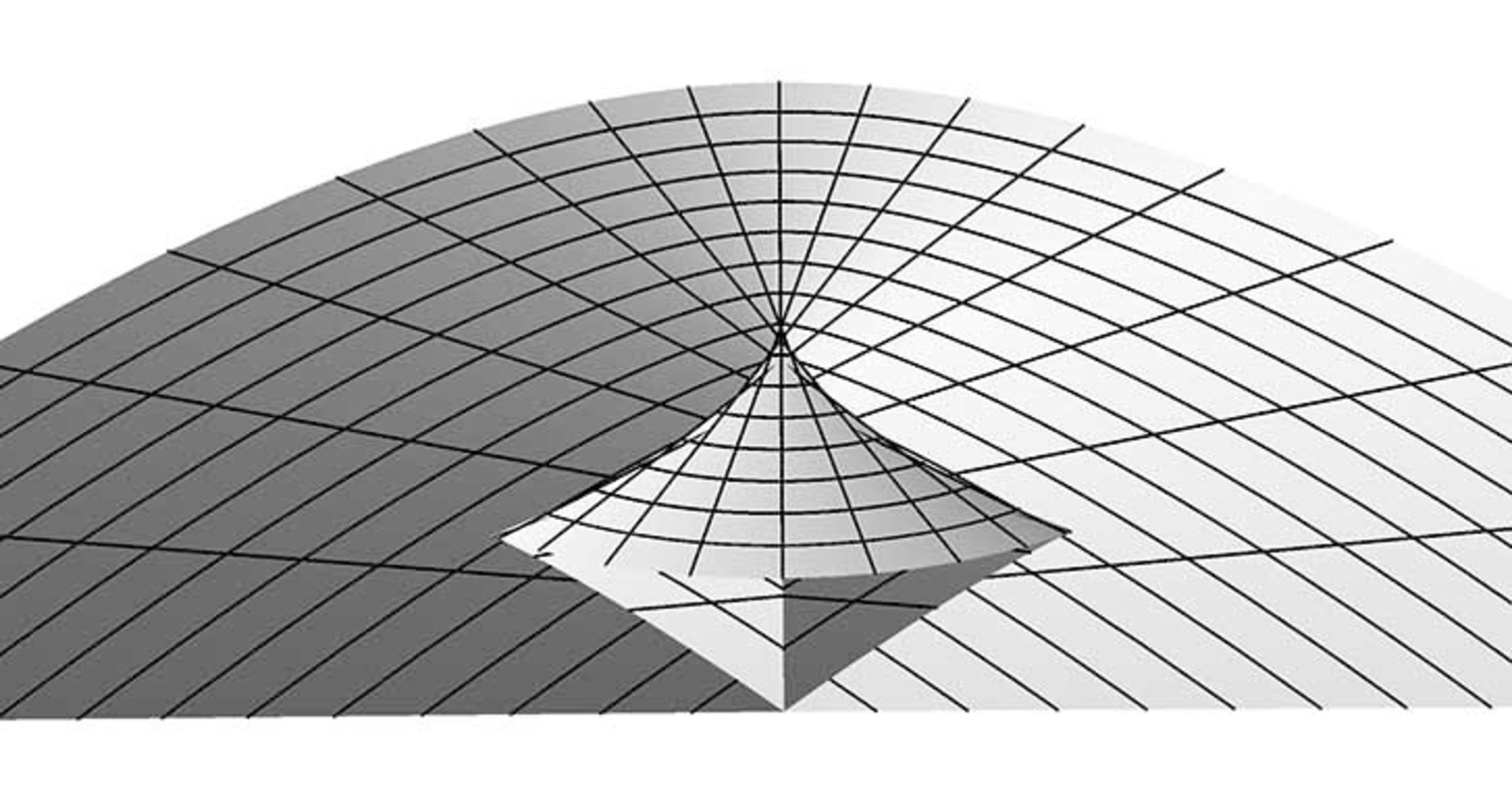} \,\,\,
  \includegraphics[height=1.7cm]{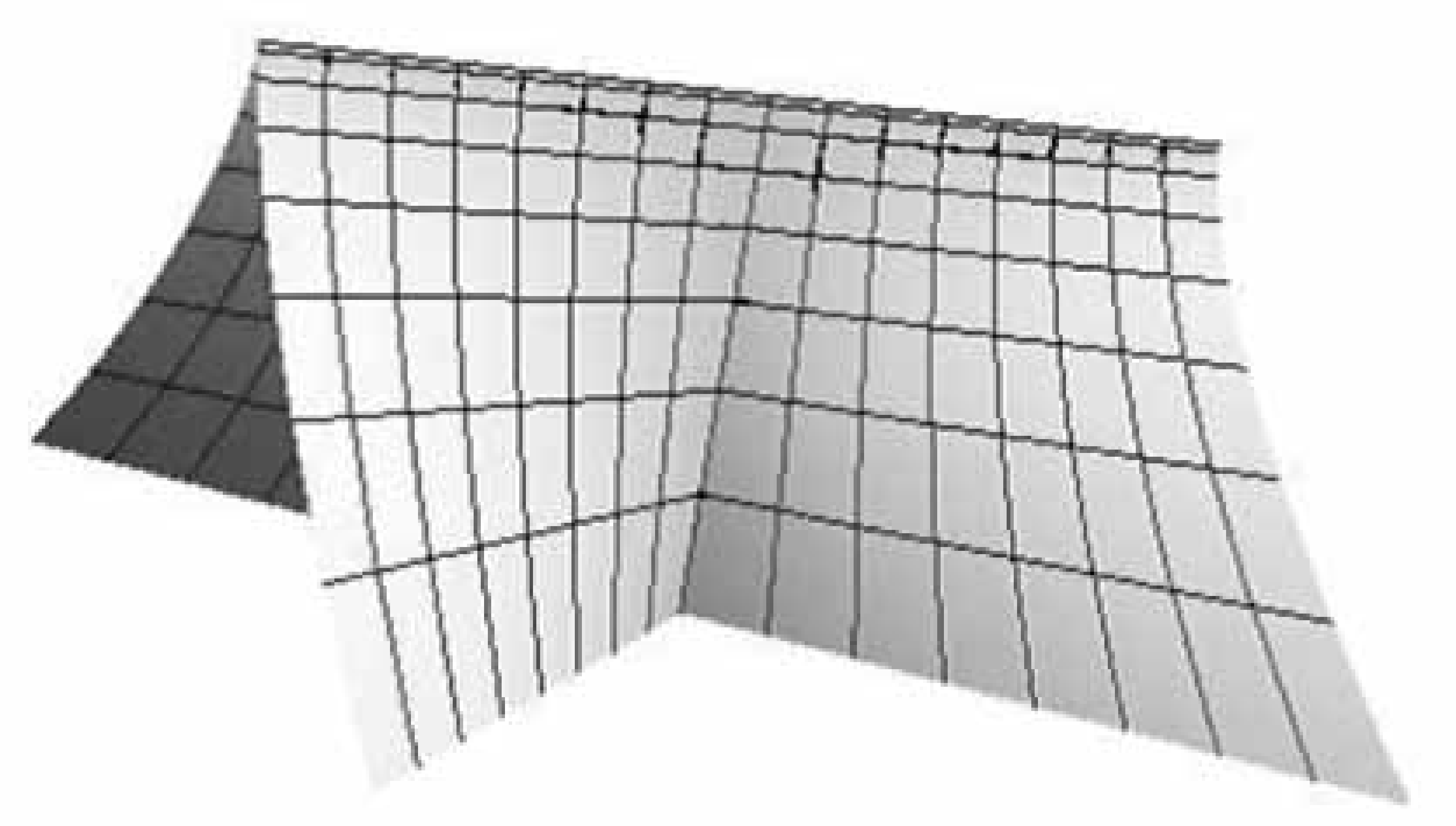}  \,\,\, 
	\includegraphics[height=2.0cm]{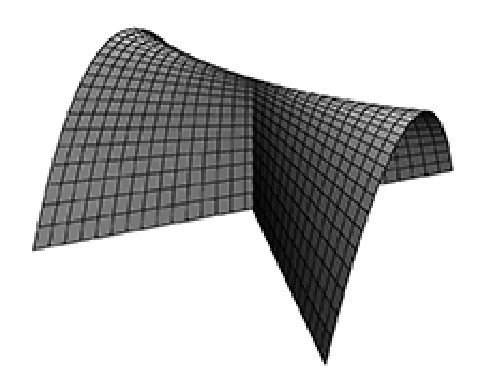}
\caption{A cuspidal edge, swallowtail, cuspidal cross cap,
cross cap,
from the left.}\label{Fig:four}
\end{center}
\end{figure}

Zakalyukin \cite{Z} pointed out that the coincidence
of the images of two wave front germs
induces the right equivalence of corresponding wave front
germs under a certain properness of the map germs.
It is then natural to ask under what  possible weaker conditions
the conclusion of Zakalyukin's lemma is still true.
In this paper, we try to give 
such a condition as follows:

Let $f:U\to \R^{n+1}$ be a continuous map
and $V$ be an open neighborhood of $p\in U$.
Then $f$ is called 
{\it $V$-proper} at $p$
if, for $\epsilon(>0)$,
there exists $r\in (0,\epsilon)$ 
such that
$(f|_V)^{-1}(\overline{B(f(p),r)})$ is a compact 
subset of $V$
(cf. 
Definition \ref{def:P2}
in Section 1),
where $B(f(p),r)$ is the open ball centered at $f(p)$ of radius $r$,
$\overline{B(f(p),r)}$ is its closure in $\R^{n+1}$ 
and $f|_V$ is the restriction of $f$ to the subset $V$.
The following assertion gives a Zakalyukin-type lemma:

\medskip
\noindent
{\bf Theorem A.}
{\it Let $U_i$ $(i=1,2)$ be a 
neighborhood of $p_i\in \R^{n}$ and let
$f_i:U_i\to \R^{n+1}$ $(i=1,2)$
be two $C^r$-frontal maps
with unit normal vector fields $\nu_i$ along $f_i$
satisfying 
\begin{itemize}
\item[(a1)] $f_1(U_1)$ is a subset of $f_2(U_2)$
and $(P:=)f_1(p_1)=f_2(p_2)$,
\item[(a2)] 
$f_2$ is 
$U_2$-proper\footnote{
We cannot drop the
condition that $f_2$ is $U_2$-proper.
In fact, the condition $f_2^{-1}(P)=\{p_2\}$
implies only the existence of a neighborhood $V(\subset U_2)$
of $p_2$ so that $f_2$ is $V$-proper (see  
Theorem~\ref{prop:I1}),
but $V$ may not coincide with $U_2$ in general.
} 
at $p_2$
and $f_2^{-1}(P)=\{p_2\}$,
\item[(a3)] the regular set of $f_i$ $(i=1,2)$
 is open dense in $U_i$,
\item[(a4)] each 
Legendrian lift $L_{f_i}$ $(i=1,2)$ is injective on a certain
neighborhood of $p_i$ $($if $f_i$ is a wave front, 
this condition is satisfied$)$.
\end{itemize}
Then there exists a homeomorphism $\psi:V_1\to V_2$
between certain connected neighborhoods $V_i$ $(i=1,2)$
of $p_i$
satisfying the following properties:
\begin{enumerate}
\item[(1)] $\overline{V_i}\subset U_i$,
\item[(2)] $f_1=f_2\circ \psi$ 
and $\nu_1=\pm \nu_2\circ \psi$ hold on $V_1$.
\end{enumerate}
Moreover, if $f_1$ and $f_2$ are wave fronts, then
$\psi$ can  be taken as a $C^r$-diffeomorphism.
}

A Zakalyukin-type lemma for wave fronts
was given in \cite{KRSUY} (see also \cite{KRSUY2}), 
which was applied to prove criteria for 
cuspidal edges and swallowtails under the 
assumption that
$f_1^{-1}(f_1(p_1))$ is finite as well as
$f_2^{-1}(f_2(p_2))$. In the above theorem,
the conclusion is obtained without any additional assumption
for $f_1$. (The standard cuspidal edge and the
standard swallowtail satisfy the condition (a2) for any choice of
an open neighborhood $U$ of the singular point $(0,0)$
(see Proposition \ref{prop:Im}).
So, to prove  the criterion for swallowtails,
Claim 1 in \cite{KRSUY2} is not needed.) 
The map $\psi$ is called
{\it the connecting map} between $f_1$ and $f_2$.
In the statement of Theorem A,  one cannot expect that $\psi$
is smooth. In fact, if
$$
f_1(t):=(t^2,t^3),\qquad f_2(t)=(t^6,t^9) \qquad (t\in \R),
$$
then the connecting map is given by $\psi(t)=t^{1/3}$
which is not a diffeomorphism at $t=0$.
The authors are mainly interested in the case $n=2$. 
In fact, a real analytic frontal in $\R^3$ usually admits
a non-trivial isometric deformation at
singular points
 (cf. \cite{NUY,HNUY, HNSUY}),
and such isometric deformations of the surfaces are
closely related to the properties of
isomers (cf. Definition \ref{def:IS})
as seen in the authors' previous work \cite{HNSUY},
and we shall discuss isomers of generalized
cuspidal edges in the final section (Section 5) in this paper. 
Here, we consider generalized cuspidal edges as follows:
We let $I$ be a closed interval 
and  fix a $C^r$-embedded curve $\mb c:I\to \R^3$,
denoting $C(:=\mb c(I))$ its image.

\begin{Def}\label{def:gC}
Let $U$ be a domain in the $uv$-plane 
$(\R^2;u,v)$
containing an interval $J\times \{0\}$ on the $u$-axis.
A $C^r$-map $f:U\to \R^3$ defined on a
domain $U$ of $\R^2$ is called a
{\it $C^r$-differentiable generalized cuspidal edge along $C$}
if $f(J\times \{0\})$ contains $C$ and
the singular set of $f$ contains  $J\times \{0\}$,
and there exist 
\begin{itemize}
\item a  diffeomorphism $\phi$ from a tubular neighborhood $V$
of $J\times \{0\}$ to the $st$-plane $(\R^2;s,t)$
satisfying $\phi(J\times \{0\})=[-1,1]\times \{0\}$,
and
\item  a diffeomorphism  $\Phi$
from a tubular neighborhood of
$C$ to $\R^3$
\end{itemize}
such that
$\Phi\circ f\circ \phi^{-1}(s,t)=(t^2, t^3\alpha(s,t),s)$
holds on $\phi(V)$.

On the other hand,
a point $p\in U$
is called
{\it $C^r$-differentiable generalized cuspidal edge point}
of a $C^r$-map $f:U\to \R^3$ defined on a
domain $U$ of $\R^2$
if there exist a 
local diffeomorphism $\psi$
satisfying $\psi(p)=(0,0)$ 
and a local diffeomorphism $\Psi$
in $\R^3$
such that
$\Psi\circ f\circ \psi^{-1}(s,t)=(t^2, t^3\alpha(s,t),s)$
holds on a neighborhood of the origin in the $st$-plane.
\end{Def}

Since $f_0(s,t):=(t^2, t^3\alpha(s,t),s)$ 
has the non-vanishing normal vector field
$$
\tilde \nu_0(s,t) 
:=
\left(-3t \alpha(s,t)-t^2 \alpha_{t}(s,t),
\,2,\,-2 t^3 \alpha_{s}(s,t)\right),
$$
generalized cuspidal edges are all frontals.
Cuspidal edges and cuspidal cross caps are
typical examples of 
generalized cuspidal edges.
However, since generalized cuspidal edges are not
wave fronts in general (e.g. cuspidal cross caps),
the smoothness of connecting maps $\psi$ 
does not follow from Theorem A directly.
We let $I$ be a closed interval 
and  fix a $C^r$-embedded curve $\mb c:I\to \R^3$,
denoting by $C(:=\mb c(I))$ its image.
As an improvement of the statement of Theorem~A
for such singular points, 
we show the following:

\medskip
\noindent
{\bf Theorem B.}
{\it 
Let $U_i$ $(i=1,2)$ be an
open subset containing a closed interval $I_i\times \{0\}$
on the $u$-axis in $(\R^2;u,v)$, 
and let
$f_i:U_i\to \R^{3}$ $(i=1,2)$
be a 
$C^r$-differentiable 
generalized cuspidal edge along 
the same embedded space curve $C$.
If 
$f_1(U_1) \subset f_2(U_2)$
then  there exist
\begin{itemize}
\item an open subset $V_i(\subset U_i)$ containing the
interval $I_i\times \{0\}$, and
\item  
a $C^r$-diffeomorphism
$\psi:V_1\to V_2$
\end{itemize}
such that
$f_1=f_2\circ \psi$ 
and $\nu_1=\pm \nu_2\circ \psi$
hold on $V_1$, where
$\nu_i$ $(i=1,2)$ 
is a unit normal vector field along $f_i$.
As a consequence, under the assumption of Theorem A,
if $p_1$ and $p_2$ are both generalized cuspidal edge points,
then the connecting map $\psi$ can be taken as 
a diffeomorphism.
}

\medskip
The corresponding assertion for cross caps is
given in the authors' previous work \cite{HNSUY2}.
To give an application of Theorem B, 
we  prepare several terminologies 
which are useful for investigating the symmetries of 
surfaces at singular points:

\begin{figure}
\begin{center}\footnotesize
 \begin{tabular}{cc}
  \includegraphics[width=1.9in]{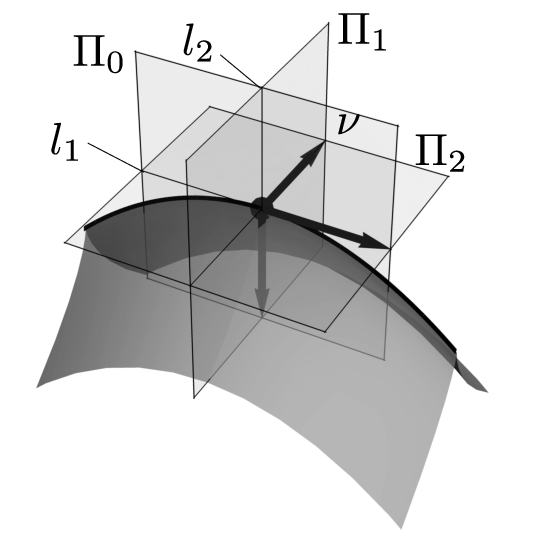} &
  \includegraphics[width=1.8in]{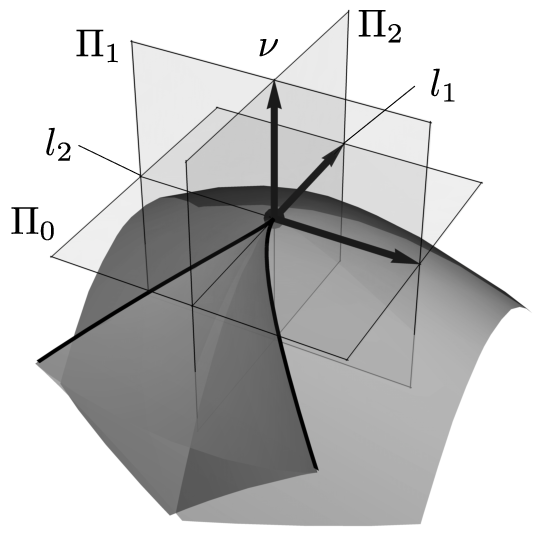}
\\
  a cuspidal edge& a swallowtail 
 \end{tabular}
\end{center}
\caption{The limiting tangent plane $\Pi_0$,  
the normal plane $\Pi_1$
and the co-normal plane $\Pi_2$
for a cuspidal edge (left)
and a swallowtail (right)
}\label{fig:frontal-planes}
\end{figure}

\begin{Def}
Let $p\in U$ be a co-rank one singular point of a frontal map
$f:U\to \R^3$.  Then there exists  a local coordinate system $(u,v)$
centered at $p$ such that $f_v(p)=\vect{0}$, and the
line
$$
l_1:=\{f(p)+t f_u(p)\,;\, t\in \R\}
$$
is defined, which is called the {\it tangent line} of $f$ at $p$
(see Figure \ref{fig:frontal-planes}).
The plane $\Pi_0$ passing through $f(p)$
which is perpendicular to the unit normal vector $\nu$ of $f$
at $f(p)$ is called the {\it limiting tangent plane}
of $f$ at $p$.
The line $l_2$ passing through $f(p)$
lying in $\Pi_0$ which is perpendicular to $l_1$
is called the {\it co-normal line} of $f$ at $p$.

On the other hand, 
the plane $\Pi_1$ passing through $f(p)$ 
which is
perpendicular to  the tangent line $l_1$ is 
called the {\it normal plane} 
of $f$ at $p$. 
Finally, 
the plane $\Pi_2$  passing through $f(p)$
spanned by the vectors $\nu(p)$ and $f_u(p)$
is called the {\it co-normal plane} at $p$.
\end{Def}

By definition, the intersection of the two planes
$\Pi_0$ and $\Pi_1$ is the co-normal line $l_2$.
When $p$ is a cuspidal edge, then 
one of the two half-lines in the co-normal line $l_2$ 
emanating from $f(p)$ points in the
direction where the image of $f$ lies, 
which is called the {\it cuspidal direction}.
The section of the image of $f$ by
$\Pi_1$ at $f(p)$
gives a cusp
(cf. \cite{F} and \cite{HNSUY}), 
which is called the {\it sectional cusp}
at $f(p)$. 
The cuspidal direction is the line
in  the normal plane $\Pi_1$ at $f(p)$
which bisects the cusp.
On the other hand, if $p$ is a swallowtail, 
then the projection of the singular set image of $f$ to 
$\Pi_0$ forms a cusp in $\Pi_0$
(cf. \cite{MSUY}).
We show the following:

\medskip
\noindent
{\bf Theorem C.}
{\it 
Let $f:U\to \R^3$ be a $C^r$-map defined on
a non-empty open subset $U\,(\subset \R^2)$, and let
$p\in U$ be a cuspidal edge,
a swallowtail, or a cuspidal cross cap.
Suppose that $f$ is $U$-proper at $p$,
$f^{-1}(f(p))=\{p\}$ and
there exist an isometry $T$ of $\R^3$ fixing $f(p)$
and an open neighborhood $V$ of $p$
such that $T\circ f(V)\subset f(U)$.
Then $T$ is an involution.
If $T$ is not the identity map, then it
is
\begin{itemize}
\item[(i)] the reflection with respect to the limiting tangent plane $\Pi_0$, 
\item[(ii)] the reflection with respect to the normal plane $\Pi_1$, 
\item[(iii)] the reflection with respect to the co-normal plane $\Pi_2$, or
\item[(iv)] the $180^\circ$-rotation with respect to the
co-normal line $l_2$.
\end{itemize}
Moreover, there exist a connected open
neighborhood $W(\subset V)$ of $p$
and a $C^r$-involution $\psi:W\to W$
such that $f\circ \psi=T\circ f$ on $W$.
Furthermore, the following assertions hold:
\begin{enumerate}
\item[(c1)] If $p$ is a cuspidal edge or a 
cuspidal cross cap singular point, 
then ${\rm (iii)}$ never happens.
Moreover, if
$p$ is a point at which
the limiting normal curvature does not vanish, 
then only ${\rm (ii)}$  happens.
\item[(c2)] 
If $p$ is a swallowtail or a cuspidal cross cap singular point, 
then 
each point of the image $f(S)$ of the self-intersection set $S(\subset W)$
of $f$ is fixed by $T$.
\item[(c3)] If $p$ is a swallowtail,
only ${\rm (iii)}$ happens.
\end{enumerate}
}

The assumption that
$f$ is $U$-proper at $p$
and $f^{-1}(f(p))=\{p\}$
is not artificial because
if a smooth map has a singular point 
giving
a cuspidal edge,
a swallowtail, or a cuspidal cross cap
singularity,
then the restriction of 
the map to a sufficiently small neighborhood
satisfies such a property (cf. Proposition \ref{prop:Im}).
The corresponding assertions for cross cap singular points
have been shown in \cite{HNSUY2}.
The  assertion (c1) contains a symmetric property of
cuspidal edges 
with non-zero limiting normal curvature,
which has been shown in \cite[Theorem 5.1]{HNSUY}
as a special case.

Also, in the authors' previous work \cite{HNSUY} 
(see also \cite{HNSUY1}),
\lq\lq isomers'' of a given real analytic cuspidal edge $f$ 
were introduced, which are
cuspidal edges with the same first fundamental form as $f$
whose singular set image coincides with that
of a given cuspidal edge $f$ but their images are
not congruent to that of $f$.
By Theorem B, we can use the fact that
image equivalence of admissible generalized cuspidal
edges is the same as right-left equivalence of them, like as in the
case of cuspidal edges.
As a consequence, almost all assertions on isomers of
real analytic cuspidal edges in \cite{HNSUY} and 
\cite{HNSUY1} can be generalized for
real analytic
admissible generalized cuspidal
edges. We will prove this fact at the end 
of this paper. Moreover, in
the authors' previous works \cite{HNSUY-O1} and \cite{HNSUY-O2},
isomers of curved foldings are also discussed,
which can be considered as analogues of isomers of cuspidal edges for
curved foldings.
We also point out the existence of a canonical map 
from the class of real analytic 
admissible generalized cuspidal
edges to the class of real analytic curved
foldings by which the isomers of
them are obtained as the image of those of the 
generalized cuspidal edge.

The paper is organized as follows:
In Section 1, we discuss properness of continuous maps
at a given point. In Section 2, we prove Theorem~A.
In Section 3, we prove Theorem B, and
Theorem C is proved in Section 4.
Finally, in Section 5,
we discuss isomers of generalized cuspidal edges
and also the connection to curved foldings.

\section{Pointwise properness for continuous maps}

There seems to be no explicit definition of
local properness of maps, not only in \cite{Z} but also
in other references as far as the authors know. 
So, in this section, we 
discuss the pointwise properness mentioned in the introduction.

Let $X, X_1$ and $X_2$ be 
locally connected
and
locally compact Hausdorff spaces.
We also fix a metric space $(Y,d_Y)$ 
all of whose closed bounded sets are compact
(e.g., the Euclidean space or a complete Riemannian manifold).
We fix 
a point $p\in X$ 
with its open neighborhood $U$.
For each $r(>0)$,
we denote by $B_Y(P,r)$ the open ball of
radius $r$ centered at $P(\in Y)$ and by $\overline{B_Y(P,r)}$
its closure.

\begin{Def}\label{def:P1}
A continuous map $f:X\to Y$ is said to be 
{\it $U$-proper at a point $p\in X$}
if there exists $r>0$
such that $(f|_U)^{-1}(\overline{B_Y(P,r)})$ is a compact 
subset of $U$.
Moreover, $f$ is said to be 
{\it strongly $U$-proper at $p$},
if
for each neighborhood $V(\subset U)$ of $p$,
there exists $r>0$ such that
$(f|_U)^{-1}(\overline{B_Y(P,r)})$ is a compact 
subset of $V$.
\end{Def}

We then give the following definition:

\begin{Def}\label{def:P2}
A continuous map $f:X\to Y$ is said to be 
{\it proper at a point $p\in X$}
if there exists a neighborhood $U$ of $p$
such that $f$ is strongly $U$-proper at $p$.  
\end{Def}

The following assertion implies that 
our pointwise properness 
defined in Definition \ref{def:P2}
can be considered as 
a property of map germs:

\begin{Proposition}\label{prop:restriction}
Suppose that $f:X\to Y$ is a strongly $U$-proper map 
at $p\in U$.
Then for each neighborhood $V(\subset U)$ of $p$,
$f$ is strongly $V$-proper at $p$.
\end{Proposition}

It is sufficient to show the following assertion:

\begin{Lemma}\label{lem:restriction}
Suppose that $f:X\to Y$ is a strongly $U$-proper map at $p\in U$.
Then, for each neighborhood $V(\subset U)$ of $p$,
there exists $r_V(>0)$ such that
\begin{align*}
& (f|_U)^{-1}(B_Y(f(p),r))=(f|_V)^{-1}(B_Y(f(p),r)), \\
&(f|_U)^{-1}(\overline{B_Y(f(p),r)})=(f|_V)^{-1}(\overline{
B_Y(f(p),r)})
\qquad (r\in (0,r_V]).
\end{align*}
\end{Lemma}

\begin{proof}
We fix a neighborhood $V(\subset U)$ of $p$.
Since $f$ is strongly $U$-proper at $p$,
there exists $r_V(>0)$ 
such that 
$$
K:=(f|_U)^{-1}(\overline{B_Y(f(p),r)})
\qquad  (r\in (0,r_V])
$$
is a compact subset of $V$.
In particular,
$
O:=(f|_U)^{-1}(B_Y(f(p),r))
$
is also a subset of $V$ for each $r\in (0,r_V]$.
So we have
\begin{equation}\label{eq:K}
O\subset (f|_V)^{-1}(\overline{B_Y(f(p),r)}),\qquad
K\subset (f|_V)^{-1}(\overline{B_Y(f(p),r)}).
\end{equation}
On the other hand, the opposite inclusions
\begin{align*}
& (f|_V)^{-1}(B_Y(f(p),r))\subset (f|_U)^{-1}(B_Y(f(p),r))=O,\\
& (f|_V)^{-1}(\overline{B_Y(f(p),r)})\subset (f|_U)^{-1}(\overline{B_Y(f(p),r)})=K
\end{align*}
are clear.
\end{proof}

Recall that a continuous map $f:X\to Y$ between two topological spaces
$X$ and $Y$ is said to be {\it proper} if for each compact subset 
$K\,(\subset Y)$, the inverse image $f^{-1}(K)$ is compact.

\begin{Example}
We consider a function $f:\R\to \R$ defined by $f(x):=x e^{-x^2}$.
This function itself is not a proper map,
but for each $\epsilon(>0)$,
the restriction of $f$ to $U:=(-\epsilon,\epsilon)$
is strongly $U$-proper at $x=0$.
\end{Example}

\begin{Example}\label{ex:0}
Define a continuous function
$
f:\R\to \R
$
so that $f(x)=x(1-|x|^{-1})$ if $|x|>1$ 
and $f(x)=0$ if $|x|\le 1$. 
Obviously, $f$ is a proper map, but not proper at $x=0$.
In fact, if we set $U:=(-\epsilon,\epsilon)$ $(0<\epsilon<1)$,
then for each $r\in (0,\epsilon)$
$$
(f|_U)^{-1}([-r,r])=(f|_U)^{-1}(\{0\})=U
$$
and so $f$ cannot be  $U$-proper at $x=0$.
This implies that the properness of a continuous map does not imply the
properness of the map at a given point, in general.
\end{Example}

\begin{Example}\label{ex:2}
Consider a continuous function given by
$$
f(x):=x \sin \frac{1}x \qquad (x\in [-1,1]).
$$
Since $[-1,1]$ is compact,
$f$ is a proper map.
Moreover, it is easy to check that $f$ is
$U$-proper at $x=0$ for each choice of 
an open interval $U:=(-\epsilon,\epsilon)$ $(0<\epsilon<1)$.
However, $f$ is not strongly $U$-proper at $x=0$.
In fact, we set 
$$
V_k:=(-a_k,a_k),\qquad a_k:=\frac{1}{k\pi},
$$
where $k$ is a positive integer satisfying $\pm a_k\in U$.
Then we have $V_k\subset U$.
We fix such a $V_k$ and set $K_r:=[-r,r]$ ($r>0$).
Since $0$ is
an interior point of $f^{-1}(K_r)$ 
and $f(\pm a_k)=0$,
there exists $\delta(>0)$ depending on $r$
satisfying $f((a_k-\delta,a_k))\subset K_r$,
which implies that $(f|_{V_k})^{-1}(K_r)$ cannot be
a compact subset of $V_k$.
\end{Example}
We prepare the following:

\begin{Lemma}\label{prop:key000}
Let $f:X\to Y$ be a continuous map and $U$ a 
neighborhood of $p\in X$.
Suppose that $f$ is $U$-proper at $p$.  
If $(f|_U)^{-1}(f(p))=\{p\}$, then $f$ is 
strongly $U$-proper at $p$.
\end{Lemma}

\begin{proof}
We fix a neighborhood $V(\subset U)$ of $p$.
Since
$f$ is $U$-proper at $p$,
there exists $r_0(>0)$ 
such that   
$$
K_r:=(f|_U)^{-1}(\overline{B_Y(f(p),r)})
$$ is a compact 
subset of $U$ for each $r\in (0,r_0]$.
It is sufficient to show that
$K_r$ is contained in $V$ for sufficiently small $r$.
If this fails, then, for each 
positive integer $k$ satisfying $1/k<r_0$,
there exists $q_k\in (f|_U)^{-1}(\overline{B_Y(f(p),1/k)})(=K_{r_0})$
which is not belonging to $V$.
By our construction of the sequence
$\{q_k\}_{k=1}^\infty$, it consists of infinitely many 
points.
Since $K_{r_0}$ is compact,
the sequence $\{q_k\}_{k=1}^\infty$ 
has an accumulation point $q_\infty\in K_{r_0}$. 
Since
$
f(q_k)\in \overline{B_Y(f(p),1/k)},
$
we have
$f(q_\infty)=f(p)$. 
Since $(f|_U)^{-1}(f(p))$ is the single point set $\{p\}$,
we can conclude $q_\infty=p$.
On the other hand,
since $q_k\in K_{r_0}\setminus V$,
we have $q_\infty\in K_{r_0}\setminus V$,
contradicting the fact $q_\infty=p$.
\end{proof}

\begin{Corollary}\label{eq:C}
Let $f:X\to Y$ be a proper map.
If $f^{-1}(f(p))=\{p\}$ holds, then
$f$ is strongly $U$-proper for each open neighborhood $U$ of $p$.
\end{Corollary}

\begin{proof}
In the setting of Lemma \ref{prop:key000},
we put $U:=X$.
Since $f$ is a proper map,
$f$ is $X$-proper at $p$.  
Since $f^{-1}(f(p))=\{p\}$ holds, 
$f$ is strongly $X$-proper at $p$.
By
Proposition \ref{prop:restriction},
$f$ is strongly $U$-proper for any open neighborhood $U$ of $p$.
\end{proof}

We next
prove the following:

\begin{Lemma}\label{prop:key}
Let $f:X\to Y$ be a continuous map and
$U$ a neighborhood of a point $p\in X$. 
If $f$  is strongly $U$-proper at $p$, then
$(f|_U)^{-1}(f(p))$ coincides with $\{p\}$.
\end{Lemma}

\begin{proof}
We set $g:=f|_U$ and suppose that
$g^{-1}(f(p))$ contains a point $q\in U$ other than $p$.
Since $X$ is a Hausdorff space,
there exists a pair $(V_1,V_2)$ of disjoint open subsets
such that $p\in V_1$ and $q\in V_2$.

Since $f$ is strongly $U$-proper at $p$,
there exists $\epsilon(>0)$
such that  $g^{-1}(\overline{B_Y(f(p),\epsilon)})\subset V_1$.
Then we have that
$$
q\in g^{-1}(f(p))\subset g^{-1}(\overline{B_Y(f(p),\epsilon)})
\subset V_1,
$$
contradicting the fact that $q\in V_2$.
\end{proof}

We next prepare the
following assertion, which is 
a generalization of
the proposition given in \cite{KRSUY}.

\begin{Proposition}\label{Lem:krsuy}
Let $f:(X,p)\to (Y,P)$ be a continuous map
such that $f^{-1}(P)$ is a finite point set.
Let $U$ be a neighborhood of $p$.
Then there exists $\delta(>0)$ such that
the connected component $V_r$ of $f^{-1}(B_Y(P,r))$
containing $p$ satisfies $\overline{V_r}\subset U$
for each $r\in (0,\delta]$. 
Moreover, $V_r$ is a relatively compact 
 open neighborhood of $p$ and
$f$ is $V_r$-proper at $p$.
\end{Proposition}

\begin{proof}
The case that $X:=U(\subset \R^n)$ is
discussed in \cite{KRSUY}.
We need a few modification 
to prove this assertion:
Since $X$ is
locally connected,
$V_r$ is an open neighborhood of $p$. Moreover,
since $X$ is a locally compact Hausdorff space, 
we can take a relatively compact 
neighborhood $W$ of
$p$ such that $\overline{W}$ is contained in $U$.
Since $f^{-1}(P)$ is a finite point set,
we  may assume that
\begin{equation}\label{eq:W}
f^{-1}(P)\cap W=\{p\}
\end{equation}
holds. As in the statement of the proposition,
we let $V_r$ be the connected component of 
$f^{-1}(B_{Y}(P,r))$ containing $p$ for each $r>0$.
Since $V_s\subset V_r$ for $s<r$, it
is sufficient to show that
$\overline{V_{1/k}}\subset W$  
for a sufficiently large integer $k>0$.
If not, there exists a point not in $W$ but lying 
in $\overline{V_{1/k}}$.
Amongst them, we would like to show
the existence of a point $p_k\in \partial W\cap \overline{V_{1/k}}$.
If not, we have the following 
decomposition
$$
\overline{V_{1/k}}=\Big(W\cap \overline{V_{1/k}}\Big)\cup
\Big((X\setminus \overline{W}) 
\cap \overline{V_{1/k}}\Big).
$$
The right-hand side is the union of
two non-empty open subsets of $\overline{V_{1/k}}$,
which is a contradiction, because
$\overline{V_{1/k}}$ is connected.
Thus, we can find such a point $p_k\in \overline{V_{1/k}}
\cap \partial W$ for each $k$.
Since $f$ is continuous,
we have
$$
f(\overline{V_{1/k}})\subset \overline{f(V_{1/k})}
\subset \overline{B_Y(P,1/k)}.
$$
By our construction, 
the sequence $\{p_{k}\}_{k=1}^\infty$ 
consists of infinitely many points,
and has an accumulation point $p_\infty\in \partial W$, because $\partial W$ is compact.
Since $p_k\in \overline{V_{1/k}}$,
$f(p_k)$ belongs to $B_Y(P,1/k)$.
In particular,
$\{f(p_k)\}_{k=1}^\infty$ converges to the point $P$,
and so $f(p_\infty)=P$ holds,
which contradicts 
\eqref{eq:W}. So we have shown that
$\overline{V_{1/k}}\subset W$  
for a sufficiently large integer $k>0$.
Since $\overline{W}$ is compact,
$\overline{V_{1/k}}$ is also compact.

We fix such an integer $k$ and set $r_0:=1/k$
and now prove that
$f$ is $V_r$-proper at $p$ under the assumption that $r<r_0$.
Since $V_r$ is a subset of $V_{r_0}$,
its closure $\overline{V_{r}}(\subset U)$ is compact.
We set $g:=f|_{V_{r}}$ and let $K$ be a compact subset of $B_Y(P,r)$.
Suppose that $g^{-1}(K)$ is not compact.
Then there exists a sequence $\{x_k\}_{k=1}^\infty$ in 
$g^{-1}(K)$ which does not accumulate 
to any point in $V_{r}$. 
Since $\overline{V_{r}}$ is compact,
$\{x_k\}_{k=1}^\infty$ 
must have an accumulation point 
$x_\infty\in \partial V_{r}$.
Since $f(x_k)\in K$,
we have
$$
f(x_\infty)\in K\subset B_Y(P,r).
$$
In particular, there exists a neighborhood $O$ of $x_\infty$
such that $f(O)\subset B_Y(P,r)$, which implies
$$
f(V_{r}\cup O)\subset B_Y(P,r).
$$
Since $x_\infty\in V_{r}\cap O$,
the union $V_{r}\cup O$ is connected.
Since $V_{r}$
is a connected component of $(f|_U)^{-1}(B_Y(P,r))$,
we have $V_{r}\cup O=V_{r}$,
contradicting the fact that $x_\infty\in \partial V_{r}$.
Thus, $g^{-1}(K)$ is a compact subset of $V_{r}$.
\end{proof}

\begin{Theorem}\label{prop:I1}
Let $f:X\to Y$ be a continuous map.
Then the following three conditions are equivalent:
\begin{enumerate}
\item The map $f$ is proper at $p$.
\item There exists a neighborhood $U$ of $p$ such that
$(f|_U)^{-1}(f(p))$
is a finite point set.
\item There exists a neighborhood $U$ of $p$ such that
$(f|_U)^{-1}(f(p))=\{p\}$, and $f$ is $U$-proper at $p$. 
\end{enumerate}
In particular, $(2)$ can be considered as
a useful criterion for the pointwise properness of 
continuous maps.
 \end{Theorem}

\begin{proof}
By Lemma \ref{prop:key}, (1) implies (2).
We set $g:=f|_U$.
Since
$g^{-1}(f(p))$
is a finite point set,
we may assume that $g^{-1}(g(p))=\{p\}$.
On the other hand,
by Proposition \ref{Lem:krsuy}, (2) implies that
$f$ is $U$-proper at $p$ for a sufficiently small
neighborhood $U$ of $p$.
So (3) is obtained.
Finally, (3) implies (1) by Lemma 
\ref{prop:key}.
\end{proof}

\begin{Corollary}\label{cor:Im}
Let $U$ be a non-empty open subset of $\R^n$
and $f:U(\subset \R^n)\to \R^N$ $(n\le N)$
a $C^r$-immersion. 
Then $f$ is proper at each point of $U$.
\end{Corollary}

\begin{proof}
Since $f$ is an immersion, it is locally injective.
So we obtain the assertion.
\end{proof}

The standard cuspidal edge,  the standard swallowtail, 
the standard cuspidal cross cap  and the standard cross cap
(see Figure 1) are defined by
\begin{align}\label{eq:std}
& f_C(u,v)=(v^2,v^3,u), \qquad f_S(u,v)=(3v^4+uv^2, 4v^3+2uv, u), \\
\nonumber
& f_{CW}(u,v)=(v^2,uv^3,u), \qquad f_{W}(u,v)=(uv,v^2,u)
\end{align}
as maps from $\R^2$ into $\R^3$, respectively.
Using these expressions, we can prove the following:

\begin{Proposition}\label{prop:Im}
The standard cuspidal edge $f_C$, 
the standard swallowtail $f_S$,
the standard cuspidal cross cap $f_{CW}$
and the standard cross cap $f_{CW}$
are $U$-proper at their singular point $o$ {\rm (which is the
origin of the domain)}
for any choice 
of an open neighborhood $U$ of $o$. 
Moreover,
$f^{-1}(f(o))=\{o\}$ holds.
\end{Proposition}

\begin{proof}
The property that
$f^{-1}(f(o))=\{o\}$ is obvious.
By Corollary \ref{eq:C},
it is sufficient to show that
$f_C$, $f_S$,
$f_{CW}$
and $f_{C}$ are  proper maps on $\R^2$.
Here we only show that $f_S$
is a proper map. (The properness of
the other maps can be proved using the same argument.)
We let $K$ be a compact subset of $\R^3$,
and let $\{(u_k,v_k)\}_{k=1}^\infty$ be a sequence in
$f_S^{-1}(K)$.
We set $f_S(u_k,v_k)=(a_k,b_k,c_k)$, then
$\{a_k\}_{k=1}^\infty$, $\{b_k\}_{k=1}^\infty$
and $\{c_k\}_{k=1}^\infty$ are bounded sequence in $\R$
because of the compactness of $K$.
Then
$\{v_k\}_{k=1}^\infty$ is bounded,
because $v_k=c_k$. 
Moreover,
the first component
of $f_S(u_k,v_k)$ 
satisfies $3u_k^4+u_k^2v_k=a_k$,
that is $u_k$ is a solution of 
the equation 
$3t^4+c_k t^2=a_k$.
Since $c_k,a_k$ are bounded,
we can conclude that
$\{u_k\}_{k=1}^\infty$ is also bounded.
Thus $\{(u_k,v_k)\}_{k=1}^\infty$ contains 
a convergent subsequence. So
$f_S^{-1}(K)$ is compact.
\end{proof}

\begin{Corollary}\label{cor:SW}
A $C^r$-map $f:U\to \R^3$
which has a  cuspidal edge, 
a swallowtail,  a cuspidal cross cap or a cross cap
singularity at $p$ is proper at $p$. 
\end{Corollary}

\begin{proof}
It is sufficient to show that
$(f|_V)^{-1}(f(p))=\{p\}$ for a sufficiently small
neighborhood $V(\subset U)$ of $p$.
However, this is obvious because 
the standard maps for cuspidal edges, 
swallowtails,  cuspidal cross caps
and cross caps have such a property.
\end{proof}

We next prove the following, 
which will be applied to prove Theorem A:

\begin{Theorem}\label{lem:I} 
Let $f_i:(X_i,p_i)\to (Y,P)$ $(i=1,2)$ be  
continuous maps $($in particular, 
$P:=f_1(p_1)=f_2(p_2))$.
Suppose that  $f_2$
is $U_2$-proper at $p_2$ and $f_2^{-1}(f_2(p_2))=\{p_2\}$.
Then the following three conditions are equivalent:
\begin{enumerate}
\item There exists a neighborhood $V_i(\subset U_i)$ of $p_i$
for each $i=1,2$ such that $f_1(V_1)\subset f_2(V_2)$.
\item There exist $r>0$ and a neighborhood $V_i(\subset U_i)$ 
of $p_i$ for each $i=1,2$
such that $f_1(V_1)\cap B_Y(P,r) \subset f_2(V_2)\cap B_Y(P,r)$.
\item For each neighborhood $V_i(\subset U_i)$ of $p_i$ $(i=1,2)$,
there exists a relatively compact neighborhood $W_i$ of $p_i$
such that $f_1(W_1)\subset f_2(W_2)$ and 
$\overline{W_i}\subset V_i$.
\end{enumerate}
\end{Theorem}

\begin{proof}
Obviously  (1) implies (2), and also (3) implies (1).
So it is sufficient to show that (2) implies (3).
So we assume (2).
We set $g_i:=f_i|_{V_i}$ ($i=1,2$).
Since $f_2$ is $U_2$-proper at $p_2$
and $g_2^{-1}(P)=\{p_2\}$,
Lemma \ref{prop:key000}
implies that
$f_2$ is strongly $U_2$-proper at $p_2$.
Hence, $f_2$ is strongly $V_2$-proper at $p_2$ by
Proposition \ref{prop:restriction}.
As a consequence,
$$
K_2:=g_2^{-1}(\overline{B_Y(P,r)})
$$
is a compact subset of $V_2$ for sufficiently small $r(>0)$.
We fix such an $r$, and
set 
$$
W_2:=g_2^{-1}(B_Y(P,r)).
$$
Then we have $\overline{W_2}\subset K_2(\subset V_2)$.
Since $X$ is a locally compact Hausdorff space,
there exists a relatively compact neighborhood $W_1$ of $p_1$
satisfying $\overline{W_1} \subset V_1$.
Moreover, since $f_1$ is continuous, we may assume
$
W_1\subset g_1^{-1}(B_Y(P,r)),
$
and so
\begin{align*}
f_1(W_1)&=g_1(W_1)
\subset B_Y(P,r)\cap g_1(V_1) \\
&\subset B_Y(P,r)\cap g_2(V_2)=g_2(W_2)=f_2(W_2).
\end{align*}
Since $\overline{W_i}\subset V_i$ ($i=1,2$), we obtain (3).
\end{proof}

\begin{Example}\label{ex:00}
In Theorem \ref{lem:I}, the assumption that
$f_2:U_2\to Y$  is proper at $p_2$ cannot be removed.
In fact, we set 
$f_1(x):=x$ ($x\in \R$) 
and let $f_2(x)$ be the function on $\R$ defined in
Example \ref{ex:0}. Then  $f_1(\R)=f_2(\R)=\R$ holds.
However, if we choose $V_1=V_2=(-1,1)$,
then $f_1(\overline{W_1})\subset f_2(\overline{W_2})$ never holds
for any choice of a pair of open intervals $(W_1,W_2)$
containing the origin in $(-1,1)$.
In this case, $f_2(x)$ is not proper at $x=0$ as shown in
Example \ref{ex:0}. 
\end{Example}

Here, we give the following terminology:

\begin{Def}
Let $f_i:(X_i,p_i)\to (Y,P)$ $(i=1,2)$ be two 
continuous maps. 
Then $f_1$ is said to be {\it image equivalent to $f_2$}
with respect to the pair of points $(p_1,p_2)$ 
if for any choice of 
a neighborhood $U_i(\subset X_i)$ ($i=1,2$)
of $p_i$, there exists a neighborhood $V_i(\subset U_i)$ ($i=1,2$)
of $p_i$ such that $f_1(V_1)\subset f_2(U_2)$
and $f_2(V_2)\subset f_1(U_1)$ hold simultaneously.
\end{Def}

Related to this definition, we also give the following.

\begin{Def}
Let $f_i:(X_i,p_i)\to (Y,P)$ $(i=1,2)$ be two 
continuous maps. 
Then $f_1$ is said to be {\it equi-image equivalent to $f_2$ as
a map germ} 
if for any choice of 
a neighborhood $U_i(\subset X_i)$ ($i=1,2$)
of $p_i$, there exists a neighborhood $V_i(\subset U_i)$ ($i=1,2$)
of $p_i$ such that $f_1(V_1)=f_2(V_2)$.
\end{Def}

\begin{Proposition}\label{prop:I2}
Let $f_i:(X_i,p_i)\to (Y,P)$ $(i=1,2)$ be two continuous maps
which are proper at $p_i$.
Then, the following assertions are equivalent each other:
\begin{enumerate}
\item The map $f_1$ is image equivalent to $f_2$
with respect to $(p_1,p_2)$.
\item The map $f_1$ is equi-image equivalent to $f_2$
with respect to $(p_1,p_2)$.
\end{enumerate}
\end{Proposition}

\begin{proof}
Since the assertions (1) and (2) are local,
we may assume that
$f_i^{-1}(P)=\{p_i\}$,
without loss of generality.
It is obvious that (2)  implies (1).
So it is sufficient to show that (1) implies (2).
We let $U_i$ ($i=1,2$) be a neighborhood of $p_i$.
Since $X_i$ is a locally compact Hausdorff space,
we can take a neighborhood $W_i$ of $p_i$ ($i=1,2$) so that
$\overline{W_i}(\subset U_i)$ is compact and 
$f_i$ is $W_i$-proper at $p_i$ (cf. Proposition \ref{Lem:krsuy}).
By Theorem~\ref{prop:I1}, we may assume that 
$f_i$ ($i=1,2$) is strongly $W_i$-proper at $p_i$.
By (1), we may assume that both
$f_2(V_2)\subset f_1(W_1)$
and
$f_1(V_1)\subset f_2(W_2)$ hold
for some $V_i\subset W_i$ ($i=1,2$),
and there exists $r_0(>0)$ such that
$$
(f_2|_{W_2})^{-1}(B_Y(P,r))\subset V_2
\quad \text{and}
\quad
(f_1|_{W_1})^{-1}(B_Y(P,r))\subset V_1
$$
for  $r\in (0,r_0]$.
Then we have
\begin{align*}
f_2\Big(
(f_2|_{W_2})^{-1}(B_Y(P,r))
\Big)
&= B_Y(P,r)\cap f_2(V_2)
\\ 
&\subset 
B_Y(P,r)\cap f_1(W_1)=
f_1\Big(
(f_1|_{W_1})^{-1}(B_Y(P,r))
\Big)
\end{align*}
and
\begin{align*}
f_1\Big(
(f_1|_{W_1})^{-1}(B_Y(P,r))
\Big)
&= B_Y(P,r)\cap f_1(V_1)
\\ 
&\subset 
B_Y(P,r)\cap f_2(W_2)=
f_2\Big(
(f_2|_{W_2})^{-1}(B_Y(P,r))
\Big).
\end{align*}
So, we have
$$
f_1\Big(
(f_1|_{W_1})^{-1}(B_Y(P,r))
\Big)
=
f_2\Big(
(f_2|_{W_2})^{-1}(B_Y(P,r))
\Big)
\quad (r\in (0,r_0]).
$$
\end{proof}

It should be remarked that
Theorem A in the introduction gives 
a sufficient condition for equi-image equivalency without assuming
the image equivalency between two maps
(see Corollary \ref{cor:EQA}).

\begin{Remark}
We consider the condition that for any choice of 
a neighborhood $U_2(\subset X_2)$ 
of $p_2$, there exists a neighborhood $V_1(\subset U_1)$ 
of $p_1$ such that $f_1(V_1)\subset f_2(U_2)$.
This condition does not
imply the existence of
a neighborhood $V_2$ such that $f_2(V_2)\subset 
f_1(V_1)$.
In fact, if we set
$$
f_1(x)=x^2,\qquad f_2(x)=x,
$$
then 
$
f_1((-r,r))\subset f_2(\R)
$,
but the opposite inclusion
$
f_2((-r,r))\subset f_1(\R)
$
never holds for any $r>0$.
So, to show  
equi-image equivalency,
we need to assume 
the image equivalency of $f_1$ and $f_2$
in the statement of
Proposition~\ref{prop:I2}.
\end{Remark}

\section{Proof of Theorem A}
To prove Theorem A, we prepare several propositions and lemmas:
Let $U_i$ ($i=1,2$) be two domains in $\R^n$,
and let $f_i:U_i\to \R^{n+1}$ be two
$C^r$-differentiable frontal maps
with $C^r$-differentiable unit normal vector fields $\nu_i$ defined on $U_i$.
Since $p_2$ satisfies (a2),
\begin{equation}\label{eq:g2}
f_2^{-1}(f_2(p_2))=\{p_2\}
\end{equation}
holds.
We take the \lq Legendrian lift\rq\ 
\[
    L_{f_i}:=(f_i,\nu_i)\colon{} U_i \longrightarrow 
 \R^{n+1} \times S^n
\]
of $f_i$ ($i=1,2$). 
We also consider the map 
\[
    L'_{f_2}:=(f_2,-\nu_2)\colon{} U_2 \longrightarrow 
 \R^{n+1} \times S^n.
\]
By (a4), there exists an
open neighborhood $U'_i(\subset U_i)$ ($i=1,2$)
such that
$L_{f_i}$ is injective on $U'_i$.
Since $(P, \nu_2(p_2))\ne (P, -\nu_2(p_2))$
($P:=f_1(p_1)=f_2(p_2)$),
there exists a relatively compact 
open subset $V_2$ of $U'_2$
satisfying  $\overline{V_2}\subset U'_2$
and
\begin{equation}\label{eq:emptyU2}
L_{f_2}(\overline{V_2})\cap L'_{f_2}(\overline{V_2})=\emptyset.
\end{equation}
By Theorem \ref{lem:I},
there exists a relatively compact 
open subset $V_1$ of $U'_1$
satisfying  $\overline{V_1}\subset U'_1$
and $f(\overline{V_1})\subset f_2(V_2)$.
We then set
$$
L_1:=L_{f_1}|_{\overline{V_1}},\quad
L_2:=L_{f_2}|_{\overline{V_2}},\quad
L'_2:=L_{f_2}|_{\overline{V_2}},
$$
and 
$$
B_+:=\{p\in \overline{V_1} \,;\, L_{1}(p)\in L_{2}(\overline{V_2})\},\quad
B_-:=\{p\in \overline{V_1}\,;\, L_{1}(p)\in L'_{2}(\overline{V_2})\}.
$$
Then we can rewrite 
$$
B_+=L_1^{-1}(L_{2}(\overline{V_2})),\qquad
B_-=L_1^{-1}(L'_{2}(\overline{V_2}))
$$
and so $B_\pm$ are closed subsets of $\overline{V_1}$.
We set $g_i:=f_i|_{\overline{V_i}}$ ($i=1,2$) and
let $\mathcal R_i$ be the set of regular values of the map $g_i$.
We set
$
\mathcal R:=\mathcal R_1\cap \mathcal R_2
$
and
$$
A_1:=g_1^{-1}(\mathcal R),\qquad S_1:=\overline{V_1}\setminus A_1
$$

\begin{Proposition}\label{lem:added}
The relation
$
\overline{V_1}= B_+\cup B_-
$
holds.
\end{Proposition}

To prove this, we prepare the following lemma:

\begin{Lemma}\label{lem:O}
$A_1$ is a dense subset of $\overline{V_1}$. 
\end{Lemma}

\begin{proof}
We suppose that $A_1$ is not dense
in  $\overline{V_1}$.
Then $S_1$ has an interior point $q$.
Since the boundary $\partial V_1$ of $V_1(\subset U_1)$ has no interior 
point and
$g_1$ gives an immersion on an open dense set of $V_1$,  
we may assume that there exists an open neighborhood $W(\subset V_1)$ of $q$
such that $W\subset S_1$ and $g_1|_W$ is an immersion.
By the Sard theorem, $f(S_1)$ is of Hausdorff dimension less than $n$,
contradicting the fact that $f_1|_W$ is an immersion.
\end{proof}

\begin{proof}[Proof of Proposition \ref{lem:added}]
We fix an arbitrary $\vect{a}\in \mc R$, and
show the inverse image
$g_i^{-1}(\vect{a})$ ($i=1,2$)
are finite point sets.
It is enough to show this for $i=1$, namely, 
showing it for $g_1$.
We assume $g_1^{-1}(\vect{a})$ is an infinite point set.
Since
$\overline{V_1}$
is compact, 
taking a sequence 
$\{q_k\}_{k=1}^\infty\subset g_1^{-1}(\vect{a})$
consisting of distinct points,
it has an accumulation point
$q\in \overline{V_1}$.
Replacing $\{q_k\}_{k=1}^\infty$ by
a suitable subsequence
if necessary, we may assume that 
it converges to $q$.
Since $f_1(q_k)=\vect{a}$, by the continuity of
$f_1$, it holds that $f_1(q)=\vect{a}$.
By the definition of $\mc R$ and because $\vect{a}\in \mc R$,
$q$ is a regular point of $f_1$.
Thus there exists a neighborhood $W$ of $q$
such that $f_1|_W$ is an embedding.
Since $\{q_k\}_{k=1}^\infty$ 
converges to $q$,
we have
$$
f_1(q_k)=\vect{a} =f_1(q),
$$
contradicting the fact that
$f_1|_W$ is injective.

Hence, there exist positive integers
$l$ and $m$ such that
\[
  g_1^{-1}(\vect{a})=\{x_1,\ldots,x_l\},\quad
  g_2^{-1}(\vect{a})=\{y_1,\ldots,y_{m}\}.
\]
Since $L_i$ are injective on $\overline{V_i}$ $(i=1,2)$,
$\nu_1(x_a)\in S^n$ $(a=1,\ldots,l)$ are mutually distinct, and
$\nu_2(y_b)\in S^n$ $(b=1,\ldots,m)$ are also mutually distinct.
Thus, the images of $g_i$ ($i=1,2$) at $\vect{a}$ are
finitely many hypersurfaces that intersect transversally
to each other.
In particular, the fact $f_1(\overline{V_1})\subset f_2(\overline{V_2})$
implies $l\le m$. Thus, changing the order appropriately,
we may assume
\begin{equation}\label{eq:zaka100}
L_{1}(x_j)
=
L_{2}(y_j)\quad\text{or}\quad L'_{2}(y_j)
\quad
(j=1,\ldots,l).
\end{equation}
Namely,
$g_1^{-1}(\vect{a})\subset B_+\cup B_-
$
holds. In particular, we have
\begin{equation}\label{eq:starAA}
A_1=\bigcup_{
\vect{a} \in{\mc R}}
g_1^{-1}(\vect{a})\subset B_+\cup B_-.
\end{equation}
Since $B_+$ and $B_-$ are closed subsets 
of $\overline{V_1}$,
by taking the closure of \eqref{eq:starAA}, Lemma~\ref{lem:O} yields 
the conclusion.
\end{proof}

We next prepare the following:

\begin{Lemma}\label{lem:l1l2}
$L_1(p_1)$ coincides with
$L_2(p_2)$ or $L_2'(p_2)$.
\end{Lemma}

\begin{proof}
Take a sequence $\{q_{k}\}_{k=1}^\infty\subset A_1$ which
converges to $p_1$.
We set $Q_{k}:=f_1(q_{k})$.
Noticing that $f_1$ is regular at $q_{k}$,
let $T_{k}$ be the tangent hyperplane of $f_1$ at $f_1(q_{k})$.
Since $f_2^{-1}(Q_{k})$ is a finite point set,
there exists a point $q'_{k}\in f_2^{-1}(Q_{k})(\subset A_2)$ such that 
the tangent hyperplane of $f_2$ at $f_2(q'_{k})$ coincides with
$T_{k}$. Then 
$L_1(q_{k})=L_2(q'_{k})$ or
$L_1(q_{k})=L_2'(q'_{k})$ holds.
Since $\{q'_{k}\}_{n=1}^\infty$ is a sequence in 
$\overline{V_2}$,
there is an accumulation point 
$q'\in \overline{V_2}$.
Then $L_1(p_1)=L_2(q')$ or
$L_1(p_1)=L_2'(q')$ holds.
In particular, $f_1(p_1)=f_2(q')$ holds.
If $L_1(p_1)=L_2(p_2)$, then 
the injectivity of $L_2$ and
\eqref{eq:emptyU2} imply $q'=p_2$.
On the other hand, when $f_2^{-1}(f_2(p))$ is a finite point set,
then \eqref{eq:g2} yields that $q'=p_2$.
So we obtain the conclusion.
\end{proof}

\noindent
\begin{proof}[Proof of Theorem A]
Replacing $\nu_2$ by $-\nu_2$ if necessary, we may assume 
$L_1(p_1)=L_2(p_2)$. By Lemma \ref{lem:O}, we have
$B_+\cup B_-=\overline{V_1}$.
By \eqref{eq:emptyU2},
$
\overline{V_1}= B_+\cup B_-
$
Since $\overline{V_1}$ is connected,
either
$\overline{V_1}=B_+$ or $\overline{V_1}=B_-$ holds.
Since $L_1(p_1)=L_2(p_2)$, 
$B_+$ is non-empty.
Thus $\overline{V_1}=B_+$ holds
and so 
$
L_1(\overline{V_1})\subset L_2(\overline{V_2}).
$
Since $L_2$ is an injective continuous map
from the compact space $\overline{V_2}$ to a
Hausdorff space, $L_2^{-1}:L_2(\overline{V_2})\to \overline{V_2}$ 
is also a continuous map.
Thus, we can define a continuous map
$\psi:\overline{V_1}\to \overline{V_2}$
by
$$
\psi:=L_2^{-1}\circ L_1:\overline{V_1} \to \overline{V_2}.
$$
By definition, it satisfies $L_1=L_2\circ \psi$, that 
is $f_1=f_2\circ \psi$ and $\nu_1=\nu_2\circ \psi$ on $\overline{V_1}$. 

Finally, since $L_1$ is injective,  
$\psi$ is an injective continuous map from
the compact space $\overline{V_1}$ to
the Hausdorff space  $\overline{V_2}$.
So it gives a homeomorphism between 
$\overline{V_1}$ and $\psi(\overline{V_1})\subset \R^n$.
By the invariance of domain (cf. \cite{GG}), 
$V'_2:=\psi(V_1)$ is a connected open subset of $\R^n$.
Thus, we have $L_1(V_1)=L_2(V'_2)$
and $\psi$ gives a homeomorphism between
$V_1$ and $V'_2$.
Replacing $V_2$ by $V'_2$, 
we obtain the relation $f_1(V_1)=f_2(V_2)$.
\end{proof}

\begin{Corollary}\label{cor:EQA}
Let $f_1$ and $f_2$ are as in Theorem A, then
these two maps are equi-image equivalent.
\end{Corollary}

\begin{proof}
Since we have shown that
$f_1=f_2\circ \psi$, which implies that
$f_1|_{V_1}$ is $V_1$-proper at $p_1$
and $f_1^{-1}(f_1(p_1))=\{p_1\}$
as well as and $f_2|_{V_2}$.
Since $f_1(V_1)=f_2(V_2)$,
Theorem~\ref{lem:I}
implies that 
$f_1$ and $f_2$ are image equivalent.
Then Proposition \ref{prop:I2}
implies that
$f_1$ and $f_2$ are equi-image equivalent.
\end{proof}

\section{Proof of Theorem B}

\subsection{The half-arc-length parameter of generalized cusps}

Let $\sigma:(a,b)\to \R^2$ be a 
$C^r$-curve defined on 
an open interval $(a,b)(\subset\R)$ where $a<b$.
A point $t=c$ on $(a,b)$
is called a {\it generalized cusp} if
$\sigma'(c)(=d\sigma(c)/dt)=\mb 0$ 
and $\sigma''(c)\ne \mb 0$.
In this situation, we can take the
inverse function $t=t(w)$
of the function $w:(a,b)\to \R$
which is $C^r$-differentiable and
satisfies
$$
w(t)^2:=\left|\int_c^t |\sigma'(u)|du\right|,
\qquad \frac{dw}{dt}>0.
$$
Then $w$ gives the half-arc-length 
parameter of the curve $\sigma$ at $u=c$.
In the case of $c=0$, using the half-arc-length
parameter $w$, the curve $\sigma$ has the following
expression $($cf.~\cite{SU} and \cite{HNSUY}$)$
\begin{equation}
\label{eq:SU}
\sigma(w)=2 \int_0^w u(\cos \lambda(u),\sin \lambda(u))du,
\quad
\lambda(w):=\frac1{\sqrt{2}}\int_0^w \mu(u)du,
\end{equation}
where $\mu(u)$ is a $C^r$-function.
Regarding this geometric meaning of $w$,
the following assertion is obvious:

\begin{Proposition} \label{prop:sigma}
Let $\sigma_i(w)$ $(w\in J)$ 
$(i=1,2)$ be two $C^r$-differentiable
generalized cusps
at $w=0$, where $J:=(-a,a)$ $(a>0)$.
Suppose that $w$ is the half-arc-length
parametrization of $\sigma_i$
for each $i=1,2$ at $w=0$.  
If $\sigma_1(J)$ coincides with 
$\sigma_2(J)$, then
either $\sigma_1(w)=\sigma_2(w)$ 
or $\sigma_1(w)=\sigma_2(-w)$ holds.
\end{Proposition}

\subsection{Smoothness of $\psi$ for generalized cuspidal edges.}\
Let $C$ be a curve $C^r$-embedded in $\R^3$ which is not closed.
To prove Theorem B, we consider the following situation:

Let $f:(U;u,v)\to \R^3$ be a $C^r$-map such that 
$U$ contains a closed interval $I\times \{0\}$ on the $u$-axis
in  the $uv$-plane $\R^2$.
We assume that
$I\times \{0\}\,(\subset  U)$ 
consists of  generalized cuspidal edge points.
Without loss of generality, we  may assume that
$\gamma(u)=(u,0)$ $(u\in I)$
giving 
the arc-length parametrization of 
$\hat \gamma:=f\circ \gamma$ 
such that $\hat \gamma(I) \subset C$.
We let $\hat \Pi_{\hat \gamma(u)}$ 
be the normal plane of $f$ for each $u\in I$.
We first prove the following:

\begin{Lemma}\label{lem:00}
There exist $\epsilon(>0)$ and
an embedding
$$
\psi:V\to U \qquad
(V:=I\times [-\epsilon,\epsilon])
$$
satisfying the following properties:
\begin{enumerate}
\item $\gamma(s):=\psi(s,0)$ 
parametrizes the singular set of $f$ such that
$\hat\gamma(s)=f\circ \gamma(s)$ gives
an arc-length parametrization of $C$.
\item for each fixed $(s,0)\in V$, the curve
$\sigma_s:t\mapsto f\circ \psi(s,t)$ parametrizes the section
of the image of $f$ by $\hat \Pi_{\hat \gamma(s)}$ 
such that
$t$ is the half-arc-length parameter
of the curve $\sigma_s$.
\end{enumerate}
\end{Lemma}

The plane curve $\sigma_s$ lying in the plane
$\hat \Pi_{\hat \gamma(s)}$ 
is called the {\it sectional cusp} at $\hat \gamma(s)$.

\begin{proof}
By the definition of generalized cuspidal edge
(cf. Definition \ref{def:gC}), there exist
\begin{itemize}
\item an open subset $U_1(\subset U)$ in the $uv$-plane
containing $I\times \{0\}$, 
\item a tubular neighborhood
$\Omega(\subset \R^3)$ of the curve  $C$ containing 
the image $f(U_1)$,
\item  
a $C^r$-diffeomorphism
$\phi:U_1\to (\R^2;x,y)$ 
giving a diffeomorphism between $U_1$ and $\phi(U_1)$ and
\item a $C^r$-diffeomorphism 
$\Phi:\Omega\to \Phi(\Omega)$ 
\end{itemize}
such that
\begin{equation}\label{eq:F0F}
\Phi\circ f\circ \phi^{-1}(x,y)=(y^2, y^3\alpha(x,y),x)\,(=:f_0(x,y)),
\end{equation}
where $\alpha$ is a $C^r$-function.
Then $x\mapsto f_0(x,0)$ gives a parametrization of the
image of the singular curve of $f_0$, 
and so we can  write
$$
f_0(x(s),0)=(0,0,x(s)) \qquad (s\in I)
$$
so that $|d\hat\gamma(s)/ds|=1$,
where $x(s)$ is a $C^r$-function satisfying $x'=dx/ds>0$.
By choosing a sufficiently small $\Omega$, we may assume that
the image $\Phi(\hat \Pi_{\hat\gamma(s)})$ ($s\in I$) 
of the normal plane of $f$
is a surface embedded in $\Phi(\Omega)$.
Thus, there exists a 
family of functions $\{g^s(x,y)\}_{s\in I}$
such that 
$$
g^s(0,0)=x(s)
$$
and the graph of the function $z=g^s(x,y)$
gives a local parametrization of  $\Phi(\hat \Pi_{\hat\gamma(s)})$.
Then, the section of the image of $f_0$ by the graph
$z=g^s(x,y)$ corresponds to the image of
the sectional cusp of $f$, 
which can be characterized by
the implicit function $F^s(x,y)=0$,
in the $xy$-plane (as the domain of definition of $g^s$)
by setting
$$
F^s(x,y):=x-g^s(y^2,y^3\alpha(x,y)).  
$$
Since
$
g^s(0,0)=x(s)
$, we have
$F^s(0,0)=0$.
Since the derivative $\partial F^s(0,0)/\partial x$
is equal to $1$,
the implicit function theorem yields that
there exists a $C^r$-function
$
x=A^s(y)
$
of $y$ which  parametrizes the set $F^s=0$,
that is,
$$
A^s(y)=g^u\Big(y^2,y^3\alpha(A^s(y),y)\Big), \qquad
A^0(0)=0
$$
hold, and
$$
\hat \sigma^s(y):=\Big(y^2,y^3 \alpha(A^s(y),y),A^s(y)\Big)
$$
gives a parametrization of the slice of 
$f_0$ by $\Phi(\hat \Pi_{\hat\gamma(s)})$.
Since 
$$
A^s(0)=g^s(0,0)=x(s),
$$
the fact $dx/ds>0$ implies that
$$
\phi_1:I\times (-\epsilon_1,\epsilon_1)\ni (s,y) \mapsto (A^s(y),y)\in \R^2
$$
is a $C^r$-diffeomorphism into the 
$xy$-plane for sufficiently small $\epsilon_1>0$, 
and the parameters $s,y$
give a  new local coordinate system of the $xy$-plane
at $(0,0)$.

Computing the derivatives of the curves 
$y\mapsto \hat \sigma^s(y)$,
we have (${}':=d/dy$)
\begin{equation}\label{eq:sigma}
(\hat \sigma^s)'(0)=\mb 0,\qquad 
(\hat \sigma^s)''(0)=(2,0,*),
\end{equation}
where $*$ means a certain value, which is not required in the
later discussions.

By setting
$
\sigma^s(y):=\Phi^{-1}\circ \hat \sigma^s(y),
$
the formula
\eqref{eq:F0F} implies that
$\sigma^u$ parametrizes the section of $f$ by 
the normal plane $\hat \Pi_{s}$.
By \eqref{eq:sigma}, $y\mapsto \sigma^s(y)$ gives the
generalized cusp at $y=0$
as the section of $f$ by the normal plane $\hat \Pi_{\hat \gamma(s)}$. 
If we set
$$
w^s(y):=\op{sgn}(y)\sqrt{|B^s(y)|},
\qquad B^s(y):=\int_{0}^y |(\sigma^s)'(t)|dt,
$$
then 
$$
\phi_2:I\times (-\epsilon_2,\epsilon_2)\ni (s,y)\mapsto (s,w^s(y))\in \R^2
$$ 
is a $C^r$-diffeomorphism into the 
$st$-plane for sufficiently small $\epsilon_2\in (0,\epsilon_1)$, 
and $t:=w^s(y)$
is the  half-arc-length parametrization of $\sigma^s$.
So if we consider the $C^r$-map given by
$$
\psi(s,t):=\phi^{-1}\circ \phi_1 \circ \phi_2^{-1}(s,t),
$$ 
then  $\psi$ gives the desired parametrization of $f$
defined on $I\times [-\epsilon,\epsilon]$ for 
sufficiently small $\epsilon>0$.
\end{proof}

\begin{proof}[Proof of Theorem B]
Without loss of generality,
we may assume that there exists a closed interval $I_i$ ($i=1,2$)
in $\R$
and
$$
I_i\ni u_i \mapsto f_i(u_i,0)\in \R^3
$$
gives the arc-length parametrization of $C$.
By replacing $u_2$ by $-u_2$,
we may assume that these two parametrizations of $C$
give the same orientation.
Then we may also assume that $I_1=I_2(=I)$ and
$$
\gamma(s):=f_1(s,0)=f_2(s,0).
$$
By Lemma \ref{lem:00}, for each $i=1,2$,
there exist a positive number $\epsilon_i$
and an embedding
$$
\psi_i:I\times [-\epsilon_i,\epsilon_i]\ni (s_i,t_i)\mapsto 
\psi_i(s_i,t_i)\in U_i
$$
such that
$(s_i,t_i)$ satisfies (1) and (2) of Lemma \ref{lem:00}.
We set $g_i:=f_i\circ \psi_i$ for $i=1,2$.
Since $f_1(U_1)\subset f_2(U_2)$,
we may assume that $\epsilon_1\le \epsilon_2$.
We denote by $\hat \Pi_{s}$ the normal plane of the curve $\gamma$ at $\gamma(s)$.
Since $t$ is the half-arc-length parameter of
each section of $g_i$ ($i=1,2$)
by the plane $\hat \Pi_{s}$,
Proposition \ref{prop:sigma}
yields that
$$
g_1(s,t)=g_2(s,e(s) t)\qquad 
(e(s) \in \{+,-\})
$$
holds at each point $(s,t)\in I\times [-\epsilon_1,\epsilon_1]$,
where $e(s)$ is a sign depending on $s$.
By the continuity of $g_1$ and $g_2$,
we can conclude that $e:=e(s)$ does not depend 
on $s$. 
In particular, the unit normal vector field
of $g_1$ coincides with that 
of $g_2$ up to a sign.
If we set
$
\phi(s,t):=(s,et),
$
then  
$
f_1\circ \psi_1=f_2\circ \psi_2\circ \phi
$ 
holds, proving the assertion.

We next show the last assertion.
Let $U_i$ $(i=1,2)$ be a 
neighborhood of $p_i\in \R^{2}$ and
$f_i:U_i\to \R^{3}$ a $C^r$-frontal map
so that $p_i$ is a generalized cuspidal
edge point satisfying 
the conditions (a1)--(a4) in Theorem A.
By Theorem A,
there exists a homeomorphism $\psi:V_1\to V_2$
between certain connected neighborhoods 
$V_i(\subset U_i)$  $(i=1,2)$
of $p_i$ satisfying 
$f_1=f_2\circ \psi$ 
and $\nu_1=\pm \nu_2\circ \psi$ on $V_1$.
There exists a regular curve
$\gamma:[-\epsilon,\epsilon]\to V_1$
so that $\gamma([-\epsilon,\epsilon])$
consists of generalized cuspidal
edge points.
Since $f_2(V_2)=f_1(V_1)$, $f_1$ and $f_2$
are both considered as a generalized cuspidal
edge along the same space curve 
$C:=f_1\circ \gamma([-\epsilon,\epsilon])$.
If we choose sufficiently small
$\epsilon(>0)$, $C$ is an embedding.
So we can apply the first part of Theorem B,
and can conclude that $\psi$ is a $C^r$-diffeomorphism.
\end{proof}

\section{Proof of Theorem C}
In this section, we prove Theorem C.

\subsection{Proof of the first part of Theorem C}
Let $f:(U;u,v)\to \R^{3}$ be a $C^r$-frontal map
and $\nu$ a unit normal vector field of $f$.
A singular point $p\in U$ of $f$ is
said to be {\it non-degenerate} if
the exterior derivative $d\lambda$
does not vanish at $p$, where
$$
\lambda:=\det(f_u,f_v,\nu).
$$
Cuspidal edges, swallowtails and cuspidal cross caps
are non-degenerate singular points on frontal maps.

We denote by $\Sigma(f)$ the singular set of $f$.
We consider the case that
$p$ is a non-degenerate singular point.
By the implicit function theorem,
there exists a regular curve $\gamma(t)$ parametrizing 
$\Sigma(f)$ near $p$ such that
$\gamma(0)=p$. 
This curve $\gamma$ is called the {\it singular curve}
and $\gamma'(0)(\ne \mb 0)$ is called the {\it singular direction} at $p$.
A non-zero tangent vector $\mb v\in T_pU$
is called a {\it null vector} of $f$ at $p$ 
if $df_p(\mb v)$ vanishes. 
Then  $p$ is called  {\it type I},
if the null-vector $\mb v$ at $p$
is linearly independent of $\gamma'(0)$.
Otherwise, $p$ is called {\it type II}.

Generalized cuspidal edges are all 
non-degenerate singular points of type I.
(In particular, cuspidal edges and cuspidal cross caps
are of type I.)
On the other hand, swallowtails are
of type II.
If $p$ is of type I, then the {\it limiting normal curvature} at $p$
is given by
(cf. \cite{MSUY})
\begin{equation}\label{eq;kappa}
\kappa_\nu(p):=\frac{\hat \gamma''(0)\cdot \nu(p)}{\hat{\gamma'}(0)
\cdot {\hat\gamma'}(0)},
\end{equation}
where $\hat \gamma(t):=f\circ \gamma(t)$.
Here, we discuss symmetries  of the standard 
cuspidal edge, swallowtail and cuspidal cross cap. 

\begin{Example}
The images of the standard cuspidal edge $f_C$ and
the standard cuspidal cross cap $f_{CW}$
(cf.~\eqref{eq:std})
are both invariant under two orthogonal transformations
fixing the origin corresponding to the following 
orthogonal matrices:
$$
T_1=\left(
\begin{array}{ccc}
 1 & 0 & 0 \\
 0 & 1 & 0 \\
 0 & 0 & -1 \\
\end{array}
\right),
\qquad
T_2:=\left(
\begin{array}{ccc}
 1 & 0 & 0 \\
 0 & -1 & 0 \\
 0 & 0 & 1 \\
\end{array}
\right).
$$
Here 
\begin{itemize}
\item $T_1$ is 
the reflection with respect to
the normal plane $\Pi_1$,
\item $T_2$ is the reflection with respect to the
limiting tangent plane $\Pi_0$, and
\item  
$T_3:=T_1\circ T_2$
is the $180^\circ$-rotation with respect to 
the co-normal line $l_2$.
\end{itemize}
\end{Example}

\begin{Example}
The image of the standard swallowtail $f_{S}$ 
(cf.~\eqref{eq:std})
is invariant under an orthogonal transformation
fixing the origin corresponding the  
orthogonal matrix $T_2$, which is
the reflection with respect to
the co-normal plane $\Pi_2$.
\end{Example}
\noindent

\begin{proof}[Proof of the first part of Theorem C]
Let $f:U\to \R^3$ be as in Theorem C.
Then
we can apply Theorem A
to the maps $f$ and $T\circ f$,
and there exists a local homeomorphism $\psi$
satisfying
\begin{equation}\label{eq:f-nu}
f\circ \psi=T\circ f,\qquad \nu\circ \psi=e T\circ \nu, \qquad
\psi(p)=p, \qquad T\circ f(p)=f(p),
\end{equation}
where $\nu$ is the unit normal vector field of $f$
and $e\in \{+,-\}$.
If 
$p\in U$ is a cuspidal edge or
a swallowtail, then, by Theorem A,  $\psi$
is a local $C^r$-diffeomorphism, because cuspdial edges
and swallowtails are wave fronts.
On the other hand, if
$p$ is
a cuspidal cross cap,
then, by Theorem B,
we can conclude that
$\psi$ is also a local $C^r$-diffeomorphism.

Without loss of generality, we may set
$f(p)=(0,0,0)$ and $T$ is an orthogonal matrix.
We can take a local coordinate system $(u,v)$ centered
at $p$ such that $f_v(p)=\mb 0$ and $f_u(p)\ne \mb 0$. 
Since $f\circ \psi=T\circ f$, the vector
$f_u(p)$ is an eigenvector of $T$.
Since $T\circ f(p)=f(p)=(0,0,0)$ and $\psi(p)=p$,
the second formula of
\eqref{eq:f-nu} implies that
$T\nu(p)=\pm \nu(p)$,
that is, $\nu(p)$ is also an eigenvector of $T$.
We consider the vector
$$
\vect{w}:=f_u(p)\times \nu(p),
$$
which points in the co-normal direction.
Since $f_u(p)$ and $\nu(p)$ are eigenvectors,
$\vect{w}$ is also an eigenvector of $T$.
Thus, we can write
$$
Tf_v(p)=\lambda_1 f_v(p),\quad T\nu(p)=\lambda_2 \nu(p),
\quad T\vect{w}=\lambda_3 \vect{w},
$$
where $\lambda_i\in \{1,-1\}$ ($i=1,2,3$).
Thus, all eigenvalues of $T^2$ are equal to $1$.
Since $T^2$ is an orthogonal matrix,
it must be the identity matrix, that is, $T$ is 
an involution.

By \eqref{eq:f-nu}, we have
$$
f\circ \psi\circ \psi=T\circ f\circ \psi=T^2\circ f=f.
$$
Since cuspidal edge has no self-intersections and
the self-intersection set of cuspidal cross caps
and swallowtails have
no interior points, we can conclude that $\psi$
is a $C^r$-involution.
Moreover, if $\psi$ is an identity map, then
the fact that $T$ is not the identity map implies that
the image of $f$ lies in a plane, which is a contradiction.
So $\psi$ is a non-trivial involution, that is, it is not
the identity map.
\end{proof}

\begin{Lemma}\label{lem:R}
Let $f:U\to \R^3$ be a $C^r$-differentiable
generalized cuspidal
edge singular point $p$.
Suppose that $T$ is
an isometry of $\R^3$ fixing $f(p)$
such that $T\circ f(V)\subset f(U)$ for 
a neighborhood $V(\subset U)$ of $p$.
Then the co-normal vector of $f$ at $p$
is a $1$-eigenvector of $T$. 
In particular, the case {\rm (iii)} of Theorem C never happens. 
\end{Lemma}

\begin{proof}
Without loss of generality, we may set
$f(p)=(0,0,0)$ and $T$ is an orthogonal matrix.
The above proof of the first part of
Theorem C can apply for generalized cuspidal edges
(cf.\ Theorem~B)
and can 
conclude that the co-normal vector $\mb v$ at $p$
is a $(\pm 1)$-eigenvector of $T$.
Without loss of generality, we may assume that
$f(s,t)$ is the parametrization of $f$
given in Lemma \ref{lem:00}.
Then $\mb v:=\partial^2 f(0,0)/\partial t^2$ points 
in the co-normal direction
that the generalized cusp 
$\sigma_0(t)$ lies in.
So we can conclude that $\mb v$
is a $1$-eigenvector.
\end{proof}

\begin{Proposition}\label{thm:main1}
Let $p\in U$ be a generalized cuspidal edge
singular point of a $C^r$-map $f:U\to \R^3$.
Suppose that there exist an isometry $T$ of $\R^3$ fixing $f(p)$
and a neighborhood $V$ of $p$ such that 
$T\circ f(V)\subset f(U)$.
If $T$ is not the identity map, then 
one of the following two cases occurs:
\begin{enumerate}
\item $T$ is the reflection with respect to 
the limiting tangent plane $\Pi_0$ at $p$,
and the singular set image of $f$ lies in $\Pi_0$.
Moreover, the limiting normal curvature of $f$ 
vanishes at $p$.  
\item $T$ is the reflection with respect to the normal plane $\Pi_1$
or the $180^\circ$-rotation with respect to the co-normal 
line $l_2:=\Pi_0\cap \Pi_1$ at $p$.
In addition, the connecting map $\psi$ is 
a $C^r$-involution interchanging the orientation of the singular curve. 
\end{enumerate}
\end{Proposition}

\begin{proof}
We may assume that $f(s,t)$
is the parametrization of $f$ as in
Lemma \ref{lem:00}.
Since $s$ is the arc-length parametrization of $C$
and $t$ is the half-arc-length parametrization of each 
sectional cusp,
as seen in the proof of Theorem B,
there exists a local $C^r$-diffeomorphism $\psi$ on a neighborhood of $p$
such that $T\circ f\circ \psi=f$ and
$$
\psi(s,t)=(e_1s,e_2t), 
$$
where $e_1,\,\, e_2\in \{+,-\}$.
Then
$
\gamma(s):=(s,0)
$
parametrizes the singular set of $f$.
By Proposition \ref{thm:main1},
the co-normal direction of $f$ at $p$
is a $1$-eigenvector of $T$.
Since $p$ is of type I, $\gamma'(0)$ is linearly
independent of the null-direction of $f$ at $p$.
We set $\hat \gamma(s):=f\circ \gamma(s)$.

We first consider the case that $e_1=+$.
In this case, we have 
$$
T\circ f\circ \gamma(s)=f\circ \psi\circ \gamma(s)
=f\circ \gamma(s), 
$$
that is, the singular points of $f$ are fixed by $T$.
In particular, the tangential direction $\hat\gamma'(0)$ 
is the $1$-eigenvector of $T$.
Since the co-normal direction at $p$
is also a $1$-eigenvector of $T$ (cf.\ Lemma \ref{lem:R}), 
the limiting tangent plane $\Pi_0$ is contained in the
fixed point set of $T$.
Since $T$ is not the identity map, $T$ must be
the reflection with respect to the limiting tangent plane $\Pi_0$.
In this situation, if the limiting normal curvature at $\gamma(s)$
does not vanish, then the Gaussian curvature takes opposite sign on 
the  two
sides of $\gamma$ (cf. \cite{MSUY} or \cite[Proposition 4]{HNUY}),  
which contradicts that  
$T$ is the reflection with respect to $\Pi_0$.
So the limiting normal curvature vanishes identically along $\gamma$.
This is the case (1).

We next consider the case that 
$e_1=-$, that is, the case that the local $C^r$-diffeomorphism $\psi$ is 
reversing the orientation of the singular curve. 
This is the case (2).
If we set $\hat \gamma:=f\circ \gamma$,
then we have $T \hat \gamma'(0)=-\hat \gamma'(0)$.
Since  the co-normal direction at $p$
is a $1$-eigenvector of $T$, 
if $\nu(p)$ is a $(-1)$-eigenvector of $T$, then 
$T$ is the $180^\circ$-rotation with respect to the co-normal line.
On the other hand, if $\nu(p)$ is a 
$1$-eigenvector of $T$, then 
$T$ is the reflection with respect to the normal plane $\Pi_1$. 
\end{proof}

\subsection{Symmetries of cuspidal edges and cuspidal cross caps}
For cuspidal edges and
cuspidal cross caps,
we can prove the following:

\begin{Proposition}\label{cor:main1} 
Let $f:U\to \R^3$ be a $C^r$-map defined on
a non-empty open subset of $\R^2$, and let
$p\in U$ be 
a cuspidal edge or
a cuspidal cross cap singular point.
Suppose that
\begin{itemize}
\item the limiting normal curvature $\kappa_\nu$
does not vanish at $p$, and 
\item 
there exists an isometry $T$ of $\R^3$ fixing $f(p)$
such that $T\circ f(V)\subset f(U)$ and $T$ is not the identity map,
where $V$ is an open neighborhood $V(\subset U)$ of $p$.
\end{itemize}
Then $T$ must be the reflection with respect to 
the normal plane $\Pi_1$,
and there exists a local $C^r$-diffeomorphism $\psi$ 
$($determined by Theorems A and B$)$
satisfy the following:
\begin{enumerate}
\item 
if $p$ is cuspidal edge singular point, then
$\psi$ is an  orientation reversing $C^r$-involution
which reverses the orientation of the singular curve,
\item
if $p$ is a cuspidal cross cap, then
$\psi$ is an orientation preserving $C^r$-involution 
which reverses the orientation of the singular curve
at $p$. Moreover,
each point of
the image of the set of self-intersections
is fixed by $T$ and is lying 
in the normal plane $\Pi_1$ near $f(p)$.
\end{enumerate}
\end{Proposition}

Before proving the proposition, 
we prepare the following:

\begin{Lemma}\label{lem:C}
Let $f:U\to \R^3$ be a $C^r$-frontal
satisfying $T\circ f \circ \psi=f$ on $U$.
Suppose that
$
\tau(t):(-\epsilon, \epsilon)\to U
$
is a $C^r$-regular curve in $U$ such that
\begin{itemize}
\item[(a)] for each $t\in (0,\epsilon)$,
there exists $t_1\in (0,\epsilon)$ such that
$f\circ \tau(t)=f\circ \tau(-t_1)$,
\item[(b)] $\tau(t)$ meets the singular set of $f$ only at $t=0$, and
\item[(c)] $f(\psi\circ \tau(t))=f\circ \tau(t)$.
\end{itemize}
Then $f\circ \tau(t)$ is a fixed point of $T$ 
for sufficiently small $|t|$.
\end{Lemma}

\begin{proof}
Without loss of generality, we may assume that $f(p)=(0,0,0)$
and $T$ is an orthogonal matrix.
We let  $ds^2$ be the first fundamental form of $f$ and
set
$$
s(t):=\int_0^t |\tau'(t)|dt
\qquad
\left(|\tau'(t)|:=\sqrt{ds^2(\tau'(t),\tau'(t))}\right),
$$
which is the arc-length of the arc $\tau([0,t])$.
By the condition (b),
$\tau'(t)$ does not vanish for each $t\ne 0$ 
sufficiently close to $t=0$.
In particular,  $t\mapsto s(t)$ is monotone increasing, and
we may consider  $s$ as a continuous parametrization of 
the curve $\tau$. By (a), we have  
$
f\circ  \tau(s)=f\circ  \tau(-s).
$
Then (c) implies that
\begin{equation}\label{eq:star2}
\psi\circ \tau(s)=\tau(s) \text{ or } \psi\circ \tau(s)=\tau(-s).
\end{equation}
Thus, we have
\begin{equation}\label{eq:f-t}
T\circ f\circ \tau(s)=f\circ \psi \circ \tau(s)
=f \circ \tau(\pm s)=f\circ \tau(s),
\end{equation}
that is, $f\circ \tau(s)$ is a fixed point of $T$.
\end{proof}

\begin{proof}[Proof of Proposition \ref{cor:main1}]
Without loss of generality, we may assume that $f(p)=(0,0,0)$
and $T$ is an orthogonal matrix.
We let $(u,w)$ be the local coordinate system centered at $p$ as in 
Lemma \ref{lem:00}. 
Since $p$ is a cuspidal edge or a cuspidal cross cap,
we can take a vector $\bf v(\ne 0)$ at $f(p)$
pointing in co-normal direction and in
the image of sectional cusp of $f$ at the same time.
So $\bf v$ is a $(+1)$-eigenvector of $T$ (cf. Lemma \ref{lem:R}).
Since the limiting normal curvature does not
vanish at $p$, the involution
$\psi$ reverses the orientation of the singular curve
(cf. Proposition \ref{thm:main1}),
$\psi(u,w)=(-u,w)$
or $\psi(u,w)=(-u,-w)$  happens.

We first consider the case that $p$ is a cuspidal edge. 
Then, the second case never occurs, because
the Gaussian curvature changes sign along 
the $u$-axis (cf. \cite[Proposition 4]{HNUY}).
So we obtain $\psi(u,w)=(-u,w)$.
In this case, the section  of the image of $f$ by
the normal plane $\Pi_1$ at $f(p)$ is a cusp.
If $T\nu(p)=-\nu(p)$ holds, then
$T$ maps a side of the cusp into the opposite side
in the plane $\Pi_1$.
However, it contradicts the fact that 
the Gaussian curvature changes sign along 
the $u$-axis.
Thus, we can conclude that $T\nu(p)=\nu(p)$,
and so $T$ is a reflection with respect to the
normal plane $\Pi_1$.
Hence (1) is obtained.

We next consider the case that $p$ is a cuspidal cross cap. 
Like as in the case of cuspidal edges,
$\psi(u,w)=(-u,w)$
or $\psi(u,w)=(-u,-w)$  happens.
However, 
by the behavior of Gaussian curvature of
$f$ (cf. \cite[Corollary 1]{HNUY}),
$\psi(u,w)=(-u,w)$ never happens,
and so  $\psi(u,w)=(-u,-w)$ holds.

We suppose $T\nu(p)=-\nu(p)$.
Then $T$ must be 
the $180^\circ$-rotation about the co-normal line.
However, in this case,  $T$  maps a point $f(u,w)$
satisfying $u,w>0$
to the point $f(u',w')$
satisfying $u'<0$ and $w'>0$
(because cuspidal cross caps have self-intersections),
but it never happens
since the sign of the
Gaussian curvature changes sign along 
the $w$-axis (cf. \cite[Proposition 4]{HNUY}).
So we have $T\nu(p)=\nu(p)$, and $T$ is 
the reflection with respect to the normal plane.

Finally, we discuss the self-intersections of $f$.
For the case of standard cuspidal cross cap $f_{CW}$
(cf. \eqref{eq:std}),
the map $\tau:v \mapsto f_{CW}(0,v)$ 
parametrizes the set of self-intersections.
Since $f$ is right-left equivalent to $f_{CW}$,
$f$ has a parametrization of the set of its self-intersections
satisfying the assumption of Lemma~\ref{lem:C}.
So each point of the image of the set of self-intersections
is fixed by~$T$, and 
(2) for cuspidal cross caps is obtained.
\end{proof}

\begin{proof}[Proof of the second part of Theorem C]
The remaining assertions in Theorem C,
except for swallowtails,
follow from Propositions
\ref{thm:main1}
and
 \ref{cor:main1}. 
\end{proof}

\begin{Example}
As shown in \cite{MS}, any germ of a cuspidal edge
is congruent to
\begin{equation}\label{eq:MS}
f(u,v)=\biggl(u,a_0(u)+v^2,b_0(u)u^2+b_2(u)uv^2+b_3(u,v)v^3\biggr),
\end{equation}
where $b_3(0,0)\ne 0$.
In the normal form for germs of a cuspidal edges,
the limiting normal curvature of $f$ at $(0,0)$ 
is non-zero if and only if $b_0(0)\ne 0$.
Moreover, $f$ admits a non-trivial symmetry at $(0,0)$
if 
$$
a_0(u)=a_0(-u),\quad b_0(u)=b_0(-u),\quad
b_2(u)=-b_2(-u),\quad
b_3(u,v)=b_3(-u,v).
$$
The normal plane is the $yz$-plane. 
\end{Example}

\begin{Example}
The map $
f(u,v):=(u,v^2,u^2+uv^3)
$
has a cuspidal cross cap at $(0,0)$
whose limiting normal curvature does not vanish.
This map has a symmetry satisfying
$$
f(-u,-v)=T\circ f(u,v),\qquad
T:=\pmt{
-1 & 0 & 0\\
0 & 1 & 0\\
0 & 0 & 1
}.
$$
\end{Example}

Other examples of cuspidal edges and
cuspidal cross caps with symmetries are in \cite[Example 6.3]{HNSUY}.

\subsection{Symmetries of swallowtails}
We have proved Theorem C except for swallowtails.
In this section, we will discuss
symmetries of swallowtails mainly, and complete the proof of Theorem C.
We first prove the following:

\begin{Lemma}\label{Prop:D}
Let $f:U\to \R^3$ be a $C^r$-frontal map,
and let $p$ be a non-degenerate singular point
satisfying $T\circ f \circ \psi=f$ on $U$
for an isometry $T$ 
and a $C^r$-involution $\psi$ on $U$.
If $\psi$ is not the identity map, then 
there exists a local coordinate system $(x,y)$ centered at $p$
satisfying the following properties:
\begin{enumerate}
\item The $x$-axis is the singular curve of $f$.
\item  If $\psi$ is an orientation preserving
local $C^r$-diffeomorphism, then 
$\psi(x,y)=(-x,-y)$.
\item  If $\psi$ reverses the orientation of 
the singular curve, then either $\psi(x,y)=(x,-y)$ 
or $\psi(-x,y)=(x,y)$ holds. 
\item If $p=(0,0)$ is of type II, 
then $\partial/\partial  x$
points in the null-direction.
\end{enumerate}
\end{Lemma}

\begin{proof}
Without loss of generality, we may assume that $f(p)=(0,0,0)$
and $T$ is an orthogonal matrix.
We fix a $C^r$-differentiable Riemannian metric $ds^2_0$ 
defined on $U$ and set
$$
ds_1^2:=\frac{ds^2_0+\psi^*ds^2_0}2.
$$
Then we have $\psi^*ds_1^2=ds_1^2$.
Since $p$ is a non-degenerate singular point,
we can take a regular curve $\gamma(s)$ 
parametrizing the singular curve such that $\gamma(0)=p$,
that is, $\gamma$ is the singular curve. 
Without loss of generality, we may assume that 
$s$ is the arc-length parameter of $\gamma$ with
respect to the metric $ds^2_1$.
Let $\xi(t)$ be the vector field 
of unit length with respect to $ds^2_1$
along the curve $\gamma(t)$ 
so that $\{\xi(t),\gamma'(t)\}$
is linearly independent in $T_{\gamma(t)}U$
for each $t$.
We let $\op{Exp}_p:T_pU\to U$ be the
exponential map of $ds^2_1$ at $p$. We set
$$ 
\Gamma(x,y):=\op{Exp}_{\gamma(x)}(y \xi(x))\in U.
$$
Then $(x,y)$ gives a local $C^r$-coordinate system
centered at $p$ such that the $x$-axis
corresponds to the singular curve. 
Since $\psi^*ds^2_1=ds^2_1$ and
$\psi$ preserves the singular curve,
we can write
$$
\psi(x,y)=(e_1x,e_2y),
$$
where $e_i\in \{+,-\}$ ($i=1,2$).
If $\psi$ is an orientation preserving isometric involution, 
we have $e_1=e_2=-1$, because
$\psi$ is not the identity map.
We next consider the case that $\psi$
is orientation reversing.
Then either $(e_1,e_2)=(1,-1)$
or $(e_1,e_2)=(-1,1)$ holds.
If $p=\gamma(0)$ is of type II,
then the tangential direction $\partial/\partial x$
of the singular curve at the origin
points in the null-direction. 
\end{proof}

We now prove the following:

\begin{Theorem}\label{thm:SW}
Let $f:U\to \R^3$ be a $C^r$-map
which 
is $U$-proper at $p$,
$f^{-1}(f(p))=\{p\}$ and
has a swallowtail singularity at $p\in U$.
If there exist an isometry $T$ of $\R^3$ fixing $f(p)$
and a neighborhood $V(\subset U)$ of $p$ 
such that $T\circ f(V)\subset f(U)$, and if
$T$ is not the identity map, then 
there exist a connected open
neighborhood $W(\subset V)$ of $p$
and a $C^r$-involution  $\psi:W\to W$
such that $f\circ \psi=T\circ f$ on $W$.
Moreover, $T$ and $\psi$ 
have the following properties:
\begin{enumerate}
\item $T$ is the reflection with respect to the co-normal plane $\Pi_2$,
\item $T$ fixes each point of the image of the
set of self-intersections of $f$, and
\item $\psi$ is the orientation reversing involution which reverses
the orientation of the singular curve.
\end{enumerate}
\end{Theorem}

\begin{proof}
Without loss of generality, we may assume that $f(p)=(0,0,0)$
and $T$ is an orthogonal matrix.
Since $p$ is a swallowtail, we may also assume that $f$
has  a unit normal vector field $\nu$ along $f$.
By Theorem A and the proof of the first part of Theorem C,
$T$ is an involution 
and there exists a connected open
neighborhood $W(\subset V)$ of $p$
and the associated non-trivial $C^r$-involution $\psi:W\to W$
satisfying $T\circ f \circ \psi=f$ on $W$.
So it is sufficient to show the remaining assertions:
The self-intersection set of the
standard swallowtail $f_S$ as in
\eqref{eq:std}
is the parabola $\tau(v):=(-2v^2,v)$ in the $uv$-plane,
and $f_S$ satisfies
the assumption of Lemma \ref{lem:C} for 
the set of self-intersection.
Since $f$ is right-left equivalent to $f_S$,
the set of self-intersections of $f$ also satisfies the 
assumption of Lemma \ref{lem:C}.
So the image of 
each point in the self-intersection set of $f$ 
is fixed by $T$, proving (2).

We project the image of the singular curve into the 
limiting tangent plane, and
then its image gives a cusp by \cite[Corollary 4.10]{MSUY},
and the line bisecting the cusp is just the  limiting
tangential direction.
Thus the  tangential direction of $f$ is the $1$-eigenvector of $T$.
Since swallowtails cannot be symmetric
with respect to the limiting tangent plane at $p$,
we have $T\nu=\nu$, that is, $\nu$ is a
$1$-eigenvector of $T$.
Since $T$ is not the identity,
the co-normal vector is a 
$(-1)$-eigenvector.
Thus, $T$ is a reflection 
with respect to the co-normal plane $\Pi_1$,
proving (1).

We now prove (3):
Let $\gamma(t)$ be a regular curve in $W$ parametrizing the
singular set of $f$ satisfying $\gamma(0)=p$.
We let $ds^2$ be the first fundamental form of $f$ and
set
$$
s(t):=\int_0^t |\gamma'(t)|dt
\qquad
\left(|\gamma'(t)|:=\sqrt{ds^2(\gamma'(t),\gamma'(t))}\right),
$$
which gives the arc-length of the arc $\gamma([0,t])$,
and  $t\mapsto s(t)$ is monotone increasing.
So we may take $s$ as a (continuous) parametrization of $\gamma$.
(Although $s(t)$ is not differentiable at $t=0$,
it does not affect the following discussion.)
Since $\psi^*ds^2=ds^2$, we have
\begin{equation}\label{eq:star3}
\psi\circ \gamma(s)=\gamma(s)\,\, \text{ or }\,\, \psi\circ \gamma(s)=\gamma(-s).
\end{equation}
We let $\tau(t)$ be the regular curve in $W$ parametrizing 
the self-intersection set of $f$.
Since $f$ is right-left equivalent to the standard swallowtail
$f_S$,
we may assume that $f\circ \tau(t)=f\circ \tau(-t)$ holds.

Let $(x,y)$ be the local coordinate system as in Lemma~\ref{Prop:D}.
Since $p$ is of type II, the $x$-axis is 
the null-direction at the origin.
We then set
$ 
\sigma(y):=(0,y).
$ 
Suppose that
$ 
\psi\circ \sigma(y)=\sigma(-y).
$ 
If we set $\hat\sigma(y):=f\circ \sigma(y)$, then we have
$ 
T\hat\sigma(y)=\hat\sigma(-y)
$.
Since $\partial/\partial x$ gives
the null-direction at $p$,
$\partial/\partial y$ does not.
Hence $\hat\sigma'(0)\ne 0$ and
$ 
T\hat\sigma'(0)=-\hat\sigma'(0)
$
hold. However, this contradicts that the
tangential direction is a $1$-eigenvector.
So we have
\begin{equation}\label{eq:3}
\psi\circ \sigma(y)=\sigma(y).
\end{equation}

Suppose that $\psi\circ \gamma(s)=\gamma(s)$  happens.
Since the tangential direction of  $\gamma(t)$ at $t=0$ 
coincides with
that of $\tau(t)$ at $t=0$,
we have
$
\psi\circ \tau'(0)=\tau'(0).
$
Moreover, since the image of $\tau$ is invariant by $\psi$, we have
$
\psi\circ \tau(t)=\tau(t).
$
Then we have
$d\psi(\tau'(0))=\tau'(0)$.
Similarly, 
\eqref{eq:3} implies that
$d\psi(\sigma'(0))=\sigma'(0)$.
Since $\tau'(0)$ and $\sigma'(0)$ are linearly independent for
the standard swallowtail, it  is so  for $f$ as well.
Thus $d\psi_p$ must be the identity map on $T_pU$.
Since the isometry of the Riemannian metric $ds^2_1$ is 
determined only by its differential $d\psi$ at $p$,
we can conclude that
$\psi$ is the identity map, a contradiction. 
So we have
$
\psi\circ \gamma(t)=\gamma(-t).
$
This implies that $\psi(x,y)=\psi(-x,y)$,
proving the assertion (3).
\end{proof}

\begin{proof}[Proof of the remaining part of Theorem C]
The remaining statements of Theorem C for swallowtails 
follow from Theorem \ref{thm:SW}.
\end{proof}

\begin{Example}
We set
$$
f(u,v):=\left(u+\frac{v^2}{2}-\frac{b^2uv^2}{2}-\frac{b^2v^4}{8},
\frac{bv^3}{3}+buv,\frac{cu^2}{2}\right)
\qquad (b\ne 0,\,\,c\in \R),
$$
which is an example of a swallowtail 
with non-zero limiting normal curvature 
given in \cite{MSUY}.
This example admits a non-trivial symmetry.
The singular set is the $v$-axis, whose image lies in 
the $xz$-plane.
\end{Example}

\section{Isomers of generalized cuspidal edges and
curved foldings}

In this section, we will generalize the results
on cuspidal edges in the authors' previous work
\cite{HNSUY} to generalized cuspidal edges
and construct a canonical map from the set of
$C^\omega$-differentiable generalized 
cuspidal edges to the set of curved foldings.

\subsection*{Isomers of generalized cuspidal edges}
We let $I:=(a,b)$ ($a<b$) be a closed interval 
and  fix a $C^r$-embedded curve $\mb c:I\to \R^3$,
denoting by $C(:=\mb c(I))$ its image.
Hereafter, we assume that the curvature function of $C$
never vanishes.

We recall the following definition of \lq \lq isomers''
of generalized cuspidal edges given in \cite{HNSUY}.

\begin{Def}\label{def:IS}
Let $f_i:U_{i}\to \R^3$ ($i=1,2$) be
two $C^r$-differentiable generalized cuspidal edges along $C$.
Then $f_2$ is called an {\it isomer} of $f_1$
if it satisfies
\begin{enumerate}
\item $f_2$ is isometric to $f_1$, 
\item $f_2$ is not right equivalent to $f_1$
as a map germ along $C$.
\end{enumerate}
In this situation,
we say that $f_2$ is a {\it faithful isomer} of $f$ if
\begin{itemize}
\item there exists a local $C^r$-diffeomorphism $\phi$
such that $\phi^*ds^2_1=ds^2_2$, and
\item the orientations of $C$ induced
by $u\mapsto f_1\circ \phi(u,0)$ 
and $u\mapsto f_2(u,0)$ 
are compatible with respect to the one
induced by $u\mapsto f_1(u,0)$,
\end{itemize}
where each $ds^2_i$ ($i=1,2$) is the first fundamental form
of $f_i$.
\end{Def}

When $C$ is not a closed curve (i.e. $\mb c(a)\ne \mb c(b)$),
as in Theorem I in \cite{HNSUY}, 
the following assertion was proved:

\begin{Fact}
Let $f$ be 
a $C^\omega$-differentiable generalized cuspidal 
edge along $C$ whose limiting normal curvature function
$($cf. \eqref{eq;kappa}$)$
does not admit any zeros.
Then there exists a faithful isomer $\check f$
$($called the dual$)$ of $f$.
\end{Fact}

By virtue of Theorem B, we can prove the following assertion,
which is the generalization of \cite[Proposition 5.1]{HNSUY}
for the case of cuspidal edges:

\begin{Proposition}\label{prop:Sec5}
Let $f:U \to \R^3$  be a $C^\omega$-differentiable generalized cuspidal edge along $C$,
whose limiting normal
curvature function does not admit any zeros.  
Then the image of $\check f$ is congruent to that of $f$ 
$($i.e. the image of $\check f$ coincides with that of $f$ by a suitable
isometry of $\R^3)$ as a map germ along $C$
if and only if it satisfies the following two 
conditions:
\begin{enumerate}
\item $C$ lies in a plane, or
\item $C$ has a positive symmetry and the
first fundamental form $ds^2_f$ has an effective
symmetry, 
\end{enumerate}
where the definition of positive symmetry 
and negative symmetry are given in
\cite[Definition 1.2]{HNSUY} and the definition of
effective symmetry is given in
\cite[Definition 0.4]{HNSUY}.
\end{Proposition}

\begin{proof}
Applying our Theorem B in the
introduction:
We suppose that $\check f$ is congruent to $f$.
By \cite[Remark 4.5]{HNSUY}, it is sufficient to consider the case
that $C$ does not lie in any plane. By Theorem B,
there exist an isometry $T$ of $\R^3$ and a $C^r$-diffeomorphism
$\phi$ defined on a neighborhood of the singular curve of $f$
such that 
\begin{equation}\label{eq:ID}
T\circ f\circ \phi=\check f.
\end{equation}
Then the remaining argument is completely parallel to the
case of cuspidal edge
as of \cite[Proposition 5.1]{HNSUY}.
\end{proof}

\subsection*{Fukui's formula for generalized cuspidal edges}
Let $\mb c(u)$ $(|u|\le I)$ be a $C^r$-regular space curve with arc-length 
parameter whose curvature function $\kappa(u)$ 
is positive everywhere and has no self-intersections.
We let 
$\theta(u)$ ($u\in I$)
and $\mu(u,v)$ ($u\in I$, $|v|<\epsilon$)
are smooth functions, where $\epsilon$ is
a positive number.
We then consider a map given by (cf. \eqref{eq:SU})
\begin{equation}\label{eq:repF2}
f(u,t):=\mb c(u)+
(A(u,t),B(u,t))\pmt{
\cos \theta(u) & -\sin \theta(u) \\
\sin \theta(u) & \cos \theta(u) 
}\pmt{\mb n(u) \\ \mb b(u)},
\end{equation}
where two functions $A$ and $B$ are 
given by
\begin{align}\label{eq:ab}
(A(u,t),B(u,t))&:=2\int_0^t v 
(\cos \lambda(u,v),\sin \lambda(u,v))
dv, \\
\nonumber
&\phantom{aaaaaaaaaa} \lambda(u,t):=\frac1{\sqrt{2}}\int_0^t \mu(u,v)dv.
\end{align}
We call such a map $f(u,v)$
a {\it normal form} of generalized cuspidal edge, which was
introduced by Fukui \cite{F} see also \cite{HNSUY}.
The following fact was known:
\begin{enumerate}
\item $f(u,v)$ is actually a generalized
cuspidal edge along $C$,
\item $\theta$ gives the {\it cuspidal angle} of $f$ along $\mb c$,
\item $f$ is a cuspidal  edge along $C$ if 
and only if $\mu(u,0)\ne 0$ for $u\in I$,
\item $\kappa_s:=\kappa\cos \theta$
is called the {\it singular curvature function},
and
$\kappa_\nu=\kappa\sin\theta$
coincides with
the {\it limiting normal curvature function} of $f$ along $C$.
\end{enumerate}

The following assertion holds:

\begin{Proposition}\label{lem:Cor:ThmB}
For a given $C^r$-differentiable generalized cuspidal edge $f$ along $C$,
there exists a normal form $g$ of
a generalized cuspidal edge along $C$
such that $g$ is right equivalent to $f$.
\end{Proposition}

\begin{proof}
This is a direct consequence of Lemma \ref{lem:00}.
In fact, the parametrization of a generalized cuspidal edge
as in Lemma \ref{lem:00} just can be written in the normal form as shown in
\cite{F} and \cite{HNSUY}.
\end{proof}

We denote by
$
{\mathbb E}^\omega(C)
$
the set of $C^\omega$-differentiable
 generalized cuspidal edges along $C$
which may not be written in the normal form,
but we assume that each $f\in {\mathbb E}^\omega(C)$
is defined on $I \times (-\epsilon, \epsilon)$ for 
some positive $\epsilon$ and $I\ni s\mapsto f(s,0)$
gives the arc-length parametrization of $C$.
Then its singular curvature function $\kappa_s$
coincides with that of its normal form.

\begin{Def}[{cf. \cite[(0.9)]{HNSUY}}]
A $C^r$-differentiable generalized cuspidal edge $f$ along $C$
is said to be {\it admissible} 
if its singular curvature function $\kappa_s$
satisfies
$$
\max_{u\in I}|\kappa_s(u)|<\min_{u\in I}\kappa(u),
$$
where $\kappa$ is the curvature function  of $C$.
\end{Def}

We denote by
$
{\mathbb E}^r_*(C)
$
the set of admissible $C^\omega$-differentiable
 generalized cuspidal 
edges along $C$ belonging to ${\mathbb E}^r(C)$

Moreover, we define the subset 
$
{\mathbb E}^r_{**}(C)
$
of
$
{\mathbb E}^r_*(C)
$
consisting of generalized cuspidal edges 
with non-vanishing singular curvature functions, that is,
$f\in {\mathbb E}^r_{**}(C)$ if and only if
$$
0<\min_{u\in I}|\kappa_s(u)|\le \max_{u\in I}|\kappa_s(u)|<\min_{u\in I}\kappa(u).
$$

As in Theorem II 
in \cite{HNSUY}, 
the following assertion holds: 

\begin{Fact}\label{fact56}
For each $f\in {\mathbb E}^\omega_*(C)$,
there exist non-faithful isomers
$f_*$ $($the inverse$)$, $\check f_*$ $($the inverse dual$)$ 
such that
if $g$ is a $C^\omega$-differentiable admissible 
generalized cuspidal edge along $C$ which is
an isomer of $f$, then $g$ is right-left 
equivalent to one of $\check f$,
$f_*$ and $\check f_*$.
\end{Fact}

\begin{Remark}
In \cite{HNSUY}, these three generalized cuspidal edges
$\check f$,
$f_*$ and $\check f_*$
associated with $f$
are defined on an open subset of the domain of definition of $f$
and their first fundamental forms coincide with that of $f$.
Then $\check f$,
$f_*$ and $\check f_*$ may not be written in normal forms
even when $f$ is a normal form. In fact, the dual $\check f$ 
of generalized cuspidal edge $f$
written in the normal form as in \cite[Example 5.3]{HNSUY}
is not written in a normal form.
\end{Remark}

By Fact \ref{fact56}
with Proposition \ref{prop:Sec5},
all the arguments in
Section 5 in \cite{HNSUY} hold not only for admissible
cuspidal edges
but also for admissible 
generalized cuspidal edge 
without any changes of proofs.
So we obtain the following theorem as a consequence:

\begin{Theorem}[{A generalization of Theorems III and IV in \cite{HNSUY}}]
Let $f$
be a $C^\omega$-differentiable
admissible generalized cuspidal edge along $C$.
Then the number of the right equivalence classes of
$f$, $\check f$, $f_*$ and
$\check f_*$ is four if and only if $ds^2_f$ has no 
symmetries $($the definition of symmetries of
$ds^2_f$ is given in \cite[Definition 0.4]{HNSUY}$)$.
Moreover, let $N_f$ be 
the number of congruence classes of 
the images of the four maps $f$, $\check f$, $f_*$ 
as map germs along $C$.
Then 
\begin{enumerate}
\item if $C$ has no non-trivial symmetries 
and also $ds^2_f$ has no symmetries,
then $N_f=4$,
\item if not the case in {\rm (1)}, it holds that $N_f\le 2$, and
\item $N_f=1$ if and only if 
\begin{enumerate}
\item $C$ lies in a plane and has a non-trivial symmetry, 
\item $C$ lies in a plane and
$ds^2_f$ has a symmetry, or
\item $C$ has a positive symmetry
and $ds^2_f$ also has a symmetry.
\end{enumerate}
\end{enumerate}
\end{Theorem}

\subsection*{Isomers of developable surfaces}

\begin{Def}
A {\it $C^r$-developable strip} along $C$ is a $C^r$-embedding 
$F:U\to \R^3$ defined on a tubular neighborhood $U$ 
of $I\times \{0\}$ in $I\times \R$ such that 
\begin{itemize}
\item $I \ni u\mapsto \mb c(u):=F(u,0)\in \R^3$ gives 
the arc-length parametrization of $C$,
\item there exists a unit vector field $\xi_F(u)$  
of $F$ along $C$ $($called a {\it ruling vector field}$)$
such that $F$ can be expressed as
$$
F(u,v)=\mb c(u)+v \xi_F(u)\qquad ((u,v)\in U),\,\, \text{and}
$$
\item the Gaussian curvature of $F$ vanishes on $U$ identically.
\end{itemize}
\end{Def}

A $C^r$-developable strip $F$ represents a map germ along $C$.
We identify this induced map germ with $F$ itself if it creates no confusion.
We denote by $\mb e(u)$, 
$\mb n(u)$ and $\mb b(u)$ the unit tangent, 
unit principal normal and unit bi-normal vector field, respectively.
With these notations,  we can express $\xi_F$ as
\begin{equation}\label{eq:f-1}
\xi_F(u)=\cos \beta_F(u)\mb e(u)+ \sin \beta_F(u)
\Big(
\cos \alpha_F(u)\mb n(u)+\sin \alpha_F(u)\mb b(u)
\Big).
\end{equation}
This $\alpha_F:I\to \R$ is called the 
{\it first angular function},
and $\beta_F:I\to \R$ is 
called the {\it second angular function} of $F$.
Here, we consider the developable strips satisfying
$0<|\cos \alpha_F|< 1$.
Then, we can choose the first angular function $\alpha_F$ so that 
\begin{equation}\label{eq:alpha_F}
0< |\alpha_F(u)|<\frac{\pi}{2} \qquad (u\in I).
\end{equation}
The fact that
the Gaussian curvature of $F$ vanishes identically 
enable us to write
\begin{equation}\label{eq:beta_f}
\cot \beta_F(u)=\frac{\alpha'_F(u)+
\tau_F(u)}{\kappa_F(u)\sin \alpha_F(u)},
\end{equation}
where $\kappa_F(u)$ and  $\tau_F(u)$ are the curvature and 
torsion functions of $\mb c(u)$.
In particular, we may assume that 
\begin{equation}\label{eq:beta_f2}
0<\beta_F(u)< \pi \qquad (u\in I).
\end{equation}
In particular,  $\xi_F(u)$ satisfies
$\xi_F\cdot \mb n>0$.
We call this $F(u,v)$ a {\it normal form} of a developable strip
along $C$. 
(This definition of normal form is the same as that in \cite{HNSUY}.)
We denote by
$
{\mathbb D}^r(C)
$
the set of $C^r$-developable strips along $C$
written in the normal form.
Comparing the normal form of generalized cuspidal edges
and the normal form of developable strips,
the following map
$$
\Phi:{\mathbb E}^{\omega}_{**}(C)\ni f 
\mapsto \Phi_f \in {\mathbb D}^{\omega}(C)
$$
is uniquely induced so that
the cuspidal angle of $f$ coincides with
the first angular function of $\Phi_f$.
The developable surface $\Phi_f$  was 
introduced in Izumiya-Saji-Takeuchi \cite[Section 5.1]{IST}
as the map producing the \lq\lq osculating developable surface'' associated with
a given cuspidal edge.
So  we  call $\Phi$ the {\it IST-map}. 

In \cite{HNSUY-O1}, for a given developable strip $F$
along $C$, its isomers
$$
\check F,\quad
F_*,\quad \check F_*
$$
are defined, which  are
developable strips along $C$
whose generating curve corresponding to $C$ is 
congruent to that of $F$ in the Euclidean plane $\R^2$.
We remark that the IST-maps have the following
nice property:

\begin{Proposition}
Let $f\in {\mathbb E}^{\omega}_{**}(C)$.
Then it holds that 
$$
\Phi_{\check f}=\check \Phi_f,\quad
\Phi_{f_*}=(\Phi_f)_*, \quad
\Phi_{\check f_*}=(\check \Phi_f)_*.
$$
\end{Proposition}

\begin{proof}
As in \cite[Corollary 3.13]{HNSUY}, 
the cuspidal angle of the dual $\check f$
is equal to $-\theta_f$.
On the other hand,  
by \cite[Page 79]{HNSUY},
the cuspidal angle $\theta_{f_*}$ of the dual $f_*$
satisfies
$$
\cos \theta_{f_*}(u)=
\frac{\kappa(u)}{\kappa(-u)}\cos \theta_{f}(u),
\qquad \theta_{f}(-u) \theta_{f_*}(u)>0
$$
on $I$. 
These two facts imply
$\Phi_{\check f}=\check \Phi_{f}$
and
$\Phi_{f_*}=(\Phi_f)_*$, respectively, since the IST-map $\Psi$ is
determined only by $\theta_f$.
Since $\check f_*=(\check f)_*$ holds, the third formula is also 
obtained.
\end{proof}

\begin{figure}[h!]
\begin{center}
\includegraphics[height=1.8cm]{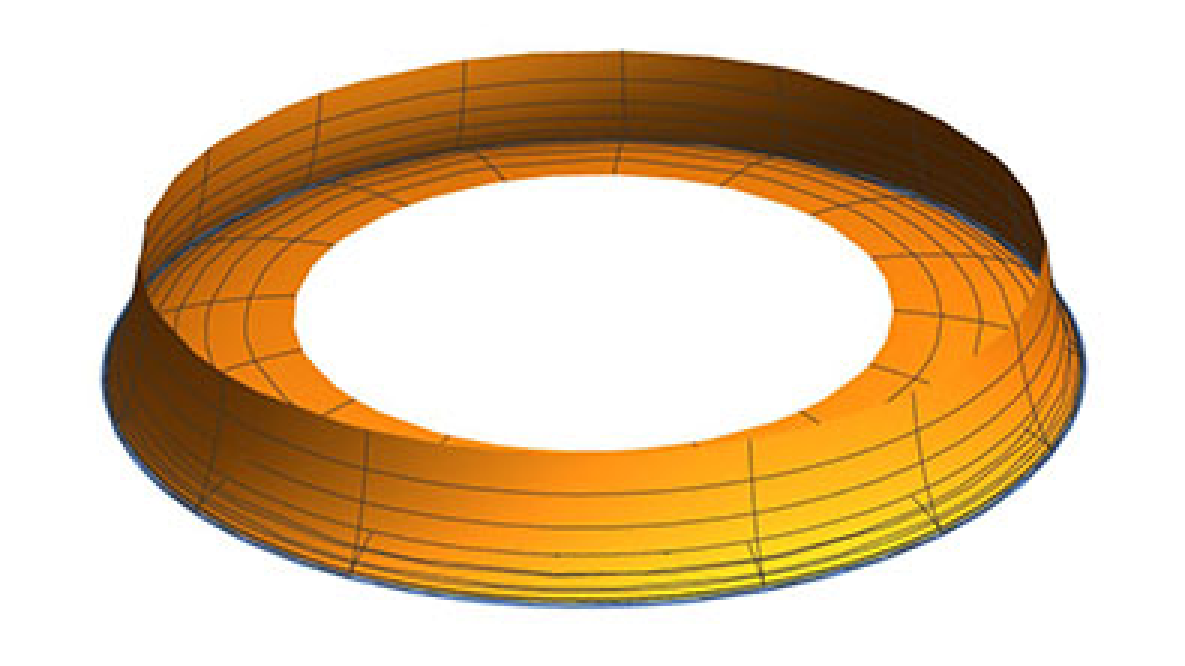}\qquad
\includegraphics[height=2.0cm]{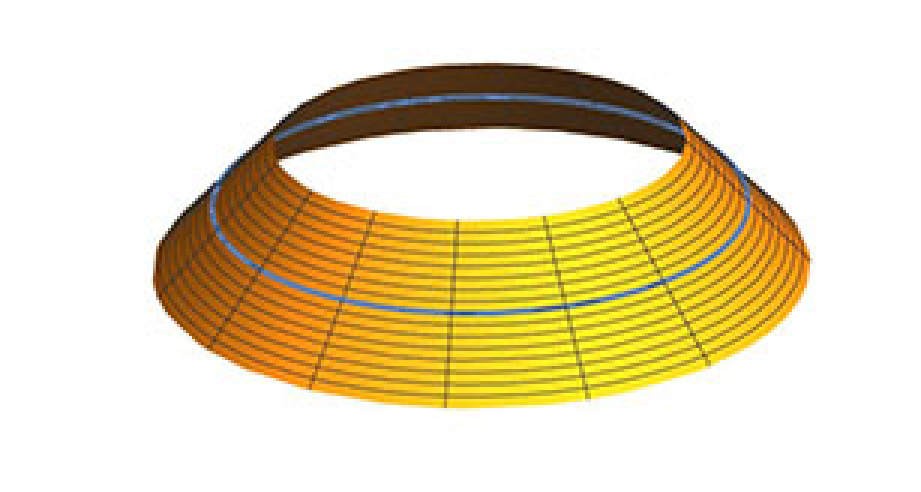}
\includegraphics[height=2.0cm]{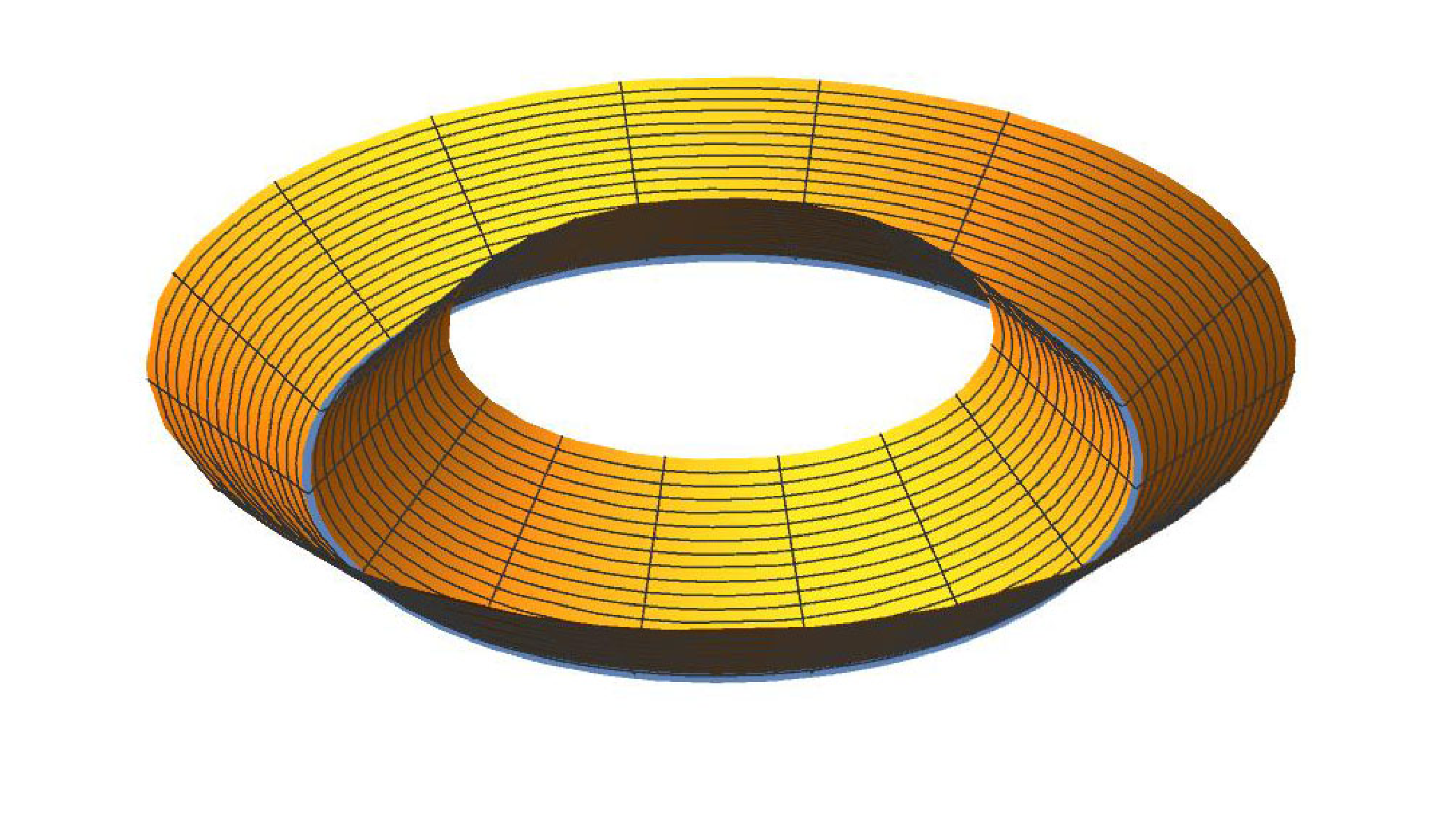}
\caption{A cuspidal edge, a
developable strip and a curved folding along 
a circle}\label{fig:2}
\end{center}
\end{figure}

A curved folding associated with a developable strip $F\in {\mathbb D}^r(C)$
is the image of the  map $\Psi_F$ defined by
$$
\Psi_F(u,v):=
\begin{cases}
F(u,v) & \text{if $u>0$}, \\
\check F(u,v) & \text{if $u<0$}.
\end{cases}
$$
It is well-known that
each  $C^r$-curved folding along $C$ 
whose absolute value of the geodesic curvature 
function (which is common in $F$ and $\check F$)
along $C$ is less than $\kappa$
can be realized as the image of a certain $\Psi_F$
for a suitable choice of $F\in {\mathbb D}^r(C)$.
In the authors' previous work, for a given 
$C^r$-curved folding along $C$,
there are three other curved foldings along $C$ whose crease pattern
is the same as the given curved folding.
By the above proposition, three curved foldings
associated to the image of $\Psi_F$ are
$\check \Psi_{F}$,
$(\Psi_{F})_*$ and
$(\check \Psi_{F})_*$.
Thus, the composition of the two maps $\Phi$ and $\Psi$
connects the isomers of 
$C^\omega$-differentiable generalized cuspidal edges to
isomers of $C^\omega$-curved foldings
(see Figure \ref{fig:2}).

\begin{acknowledgments}
The authors thank  the referee,
Toshizumi Fukui and Wayne Rossman for helpful comments.
\end{acknowledgments}

\end{document}